\documentclass[11pt]{amsart}
\usepackage{graphicx}
\usepackage{amssymb, latexsym, color, url}

\newtheorem{lemma}{Lemma}[section]
\newtheorem{theorem}{Theorem}
\numberwithin{equation}{section}
\newtheorem{proposition}[lemma]{Proposition}
\newtheorem{definition}[lemma]{Definition}

\newtheorem{remark}[lemma]{Remark}

\newcommand{\R}{\mathbb{R}}
\newcommand{\N}{\mathbb{N}}
\newcommand{\PP}{\mathbb{P}}

\newcommand{\Z}{\mathbb{Z}}

\newcommand{\C}{\mathbb{C}}

\newcommand{\e}{\varepsilon}
\newcommand{\ep}{\varepsilon}
\newcommand{\calH}{\mathcal{H}}

\newcommand{\calL}{\mathcal{L}}

\newcommand{\rest}{{\, \bf{{  \llcorner}}\,}}
\newcommand{\ds}{\displaystyle} 

\newcommand{\barint}{\overline{\hspace{.65em}}\!\!\!\!\!\!\int}

\newcommand{\calD}{\mathcal{D}}

\newcommand{\beq}{\begin{equation}}
\newcommand{\eeq}{\end{equation}}
\newcommand{\bea}{\begin{eqnarray}}
\newcommand{\eea}{\end{eqnarray}}
\newcommand{\beas}{\begin{eqnarray*}}
\newcommand{\eeas}{\end{eqnarray*}}

\definecolor{darkgreen}{rgb}{0,0.55,0}

\begin{document}

\newtheorem{conjecture}{Conjecture}
\theoremstyle{definition}
\newtheorem{example}{Example}

\title[Sobolev spaces of isometric immersions]
{Sobolev spaces of isometric immersions of arbitrary dimension and co-dimension}
\author{Robert L. Jerrard \and Mohammad Reza Pakzad}
\address{Department of Mathematics, University of Toronto,
Toronto, Ontario, Canada}\email{rjerrard@math.toronto.edu}
\address{Department of Mathematics, University of Pittsburgh,
Pittsburgh, PA}\email{pakzad@pitt.edu}
\begin{abstract}
We prove the $C^{1}_{loc}$ regularity  and developability of $W^{2,p}_{loc}$
isometric immersions of $n$-dimensional flat domains into $\R^{n+k}$ where $p\ge \min\{2k, n\}$.
We also prove similar rigidity and regularity results for scalar functions
of $n$ variables for which the rank of the Hessiam matrix is
{\em a.e.}  bounded by some $k<n$, again assuming
$W^{2,p}_{loc}$ regularity for $p\ge \min \{2k,n\}$. In particular  this includes
results about the
degenerate Monge-Amp\`ere equation, $\det D^2 u = 0$,
corresponding to the case $k=n-1$.
\end{abstract}

\maketitle
\vspace{.3in}
\tableofcontents

\section{Introduction}

\subsection{Background}
 
The question of rigidity vs. flexibility of isometric immersions has been studied in differential geometry since the end of 19th century. 
It was already known, as established by Darboux, among others, that smooth surfaces in the three dimensional space  which are isometric to a piece of plane are 
{\it developable}, i.e. they are locally  foliated as a ruled surface by straight segments aligned at each point in one of the principal directions. New developments in the mid-20th century highlighted the very fact that this {\it rigidity} statement relies strongly on the regularity of the surface. In particular, it followed from the results of Nash \cite{Nash2} and Kuiper \cite{kuiper} that there exist many $C^1$ isometric embeddings of a given flat $n$-dimensional domain  into $\R^{n+1}$ (and hence into $\R^{n+k}$ for any $k\ge 1$) with arbitrarily small upper bound on the diameter of the image, a property which rules out the developability of the image. On the other hand, 
the  developability of co-dimension one  isometric immersions of flat $n$-dimensional domains
was essentially established by Chern and Lashof \cite[Lemma 2]{ChL}
and  
Hartman and Nirenberg  \cite[Lemma 2]{HartmanNirenberg}, 
who also provided more detailed results in the case $n=2$ of surfaces.
In \cite{VWKG}, a generalized developability result for $C^2$ isometric immersions of 
a Euclidean domain $\Omega\subset \R^n$
into Euclidean spaces $\R^{n+k}, k<n$ was established.

A natural question arises, which consists in asking what would be the critical regularity threshold at which the distinction between {\it rigidity} and {\it flexibility} \'a 
la Nash and Kuiper is withheld. The most straightforward path would be to discuss this question for H\"older regular isometries of class $C^{1,\alpha}$, $0<\alpha<1$. 
Some progress is made in this direction, but the problem of the critical value of  $\alpha$ is still open. While a careful analysis of the iteration methods of Nash and Kuiper 
have lead to flexibility results for surfaces for $\alpha < 1/13$ \cite{bori2} and then for $\alpha<1/7$ \cite{CDS},  it has only been established that $C^{1,\alpha}$ 
isometric immersions of 2 dimensional flat domains into the three dimensional space are rigid if $\alpha > 2/3$ \cite{bori1, bori2, CDS}. 
In a different but related vein, Pogorelov showed   that  $C^1$ surfaces with
total zero curvature  are developable \cite[Chapter II]{Po 56} and \cite[Chapter IX]{Po 73}.  
If one only assumes H\"older regularity, it seems there is no 
consensus on what the critical exponent should be, 
as it has been conjectured to be $\alpha=1/3, 1/2$ or $2/3$.

One could also consider other  function spaces which lie somewhat below $C^2$. 
In particular, Sobolev isometries 
arise in the study of nonlinear elastic thin films. Kirchhoff's plate model 
put forward in the 19th century \cite{kirchhoff} consists in minimizing the $L^2$ norm of the second fundamental form of isometric
immersions of a 2d domain into $\R^3$ under suitable forces or boundary conditions. In other words, using the modern terminology, 
the space of admissible maps for this model is that of $W^{2,2}$ isometric immersions (See also \cite{FJMgeo, lepa_noneuclid}).  

Quite strong results are known about regularity and rigidity of {\em codimension 1} isometric immersions, as summarized in the folllowing

\begin{theorem}\label{0deve-regu}
Let $U\in W^{2,2}(\Omega, \R^{n+1})$ be an isometric immersion, where $\Omega$ is a bounded Lipschitz domain in
$\R^n$. Then $U\in C_{\rm loc}^{1,1/2}(\Omega, \R^{n+1})$. Moreover, for
every $x\in \Omega$, either $D U$ is constant in a neighborhood of $x$, or
there exists a unique $(n-1)$-dimensional hyperplane $\PP\ni x$ of $\R^n$ such
that $D U$ is constant on the connected component of $x$ in $\PP\cap \Omega$.
\end{theorem}

This was proved in by Liu and Pakzad \cite{LP}, and 
followed earlier results \cite{Pak} of the second author that established 
the $n=2$ case of Theorem \ref{0deve-regu}, 
drawing on work of Kirchheim in  \cite{kirchheimthesis} on $W^{2,\infty}$ solutions to degenerate Monge-Amp\`ere equations, discussed below.

The  result is optimal is the sense that it fails for $W^{2,p}$ isometries with $p<2$.

\begin{remark} 
In \cite{MuPa2} it was established for $n=2$ that the $C^1$ regularity can be extended up to the boundary 
if the domain is of class $C^{1,\alpha}$. This does not hold true anymore for merely $C^1$ regular domains. 
\end{remark} 

Isometric imersions of flat domains are closely related to the degenerate Monge-Amp\`ere
equation
\begin{equation}\label{MA-2d}
\ds {\rm det} (D^2u) =0 \quad \mbox {a.e.  in} \quad  \Omega,
\end{equation}
or more generally to the Hessian rank inequality
\begin{equation}\label{analytic}
\displaystyle \mbox{rank} (D^2 u) \le k \quad \mbox{a.e. in} \,\, \Omega.
\end{equation}   
This is equivalent to the degenerate Monge-Amp\`ere equation when $k=n-1$, but
for $k<n-1$ is a stronger condition. As we recall in Section \ref{S:6}, it is satisfied by the components $U^m$ of an  isometric immersion $U: \Omega \to \R^{n+k}$ of co-dimension $k$ (see Proposition \ref{isom.-is-degenerate}), and many rigidity properties of isometric immersions can be deduced
solely from the weaker condition \eqref{analytic}.

In order to discuss Sobolev solutions with lower regularity than the assumptions of the above theorem, 
it is helpful to study distributional and measure theoretic variants of condition \eqref{MA-2d}
including (in 2-dimensional domains) 
\begin{equation} \label{wMA-2d}
\ds {\rm Det} (D^2u) := -\frac 12 {\rm curl}^T {\rm curl} (Du \otimes Du) =0
\end{equation} 
for $u\in H^1(\Omega)$; or
\begin{equation}
\int_\Omega \phi_{x_1}(x, Du) u_{x_k x_2} - \phi_{x_2}(x,Du) u_{x_k x_1} \ dx = 0
\quad\mbox{ for all }\phi\in C^\infty_c(\Omega\times \R^2)\mbox{ and }k=1,2
\label{swMA-2d}\end{equation} 
for $u\in W^{2,1}(\Omega)$. Both of these imply \eqref{MA-2d} if 
$u\in W^{2,2}_{loc}(\Omega)$.
It turns out that  \eqref{MA-2d},
even in the weak form \eqref{swMA-2d},  is strong enough to imply rigidity, as shown in
the following result.

\begin{theorem}
Let $\Omega$ be a bounded, open subset of $\R^2$.

If $u\in W^{2,2}_{loc}(\Omega)$ and $\det D^2 u=0$ a.e. in $\Omega$, 
then $u\in C^1(\Omega)$ and for every point $x\in \Omega$, there exists either a neighborhood of $x$, 
or a segment passing through $x$ and joining $\partial \Omega$ at  both ends, 
on which $Du$ is constant.

More generally, the same conclusions hold if we merely assume that $u\in W^{2,1}(\Omega)$ and $u$ satisfies 
\eqref{swMA-2d}.
\label{MArigid}\end{theorem}

 Theorem \ref{MArigid} was  established for $u\in W^{2,2}_{loc}(\Omega)$ by
the second author in \cite{Pak}, see also Kirchheim \cite{kirchheimthesis}.
The final assertion of the theorem, concerning $W^{2,1}$ functions, 
is in fact a special case of a more general
result  from \cite{J2010},
that applies in the (larger) class of Monge-Amp\`ere functions, introduced by Fu \cite{fu1}
and developed  in \cite{Je1, J2010}.
If one considers not the distributional condition \eqref{swMA-2d} but just the
pointwise Monge-Amp\`ere equation \eqref{MA-2d}, then the $W^{2,2}$ hypothesis
of \cite{Pak} is optimal. Indeed,
conic solutions to \eqref{MA-2d} exist 
if the regularity is assumed to be only $W^{2,p}$ for $p<2$ (see Example \ref{ex2} below).  
One could even construct more sophisticated solutions by gluing these conic singularities in a suitable manner, using Vitali's covering theorem (Example \ref{ex-bes}).
Furthermore, 
Liu and Mal\'y \cite{LM-private} have established the existence of strictly convex $W^{2,p}$ solutions to \eqref{MA-2d} 
(but not to \eqref{wMA-2d})  when $p<2$.  In the meantime, it is known \cite{FM}  that for $p<2$, $W^{2,p}$ solutions to 
\eqref{wMA-2d} exist which are not $C^1$  and fail to satisfy the developability statement of 
Theorem \ref{MArigid} at a given point in the domain.

\medskip 

What interests us in this paper are regularity and rigidity results 
in the manner of Theorems \ref{0deve-regu} and \ref{MArigid} 
for arbitrary $1\le k < n$,
under  Sobolev regularity assumptions. 
We note that the case $k=0$ is trivial, and that there is no rigidity
whenever $k\ge n$, see for example \cite{VWKG}.

The proof  in \cite{LP} of Theorem \ref{0deve-regu} was based on induction on the dimension 
of slices of the domain and careful and detailed geometric arguments, applying 
the $W^{2,2}_{loc}$ case of Theorem \ref{MArigid} to two dimensional slices.   
These methods cannot be adapted to the 
solutions of \eqref{analytic} even for $k=1$, since one loses some natural advantages when working with (\ref{analytic}) rather than with the isometries themselves as
done in  \cite{LP}: the solution $u$  is no more Lipschitz and being just a scalar function, one loses the extra information  derived from the length 
preserving properties of isometries. On the other hand, contrary to the case of  $k=1$, regularity and developability of the Sobolev solutions to 
(\ref{analytic}) does not directly lead to the same results for the corresponding isometries (see \cite{Pak}). 

Hence, the problems of regularity and developability of Sobolev isometric immersions of  co-dimension higher than 1, and also of the developability of Sobolev 
solutions to \eqref{analytic} for  $k>1$, are more involved and could not be tackled 
through the methods discussed in \cite{Pak, LP}.  In this paper, we adapt methods of geometric measure theory, applied  by the first author in 
\cite{Je1, J2010} to the class of Monge-Amp\`ere functions,
to overcome the above obstacles for $k>1$ and tackle both of 
the isometry and rank problems for Sobolev 
regular solutions simultaneously.

\begin{remark}
 It was proved furthermore in \cite{Pak} 
that any $W^{2,2}$ isometry on a convex 2d domain can  be approximated in strong norm by smooth isometries.  The convexity assumption can be weakened 
to e.g. piece-wise $C^1$ regularity of the boundary, see also \cite{Ho0, Ho1, Ho2}.  
A generalization of these results to the co-dimension one case were obtained in \cite{LP}. It could be expected that the results of this paper could 
help in proving similar density statements in higher co-dimensions, but  that would be more technically challenging than the previous cases.
 \end{remark}

\subsection{Main results} 

We first introduce a few fundamental definitions.

\begin{definition} \label{k-planes} 

Let $\Omega \subset \R^n$ be an open set and $j\in \{1, \cdots, n\}$. We say the set $P\subset \Omega$ is a {\em $j$-plane in $\Omega$}
whenever $P$ is the connected component of the intersection of $\Omega$ and a $j$-dimensional affine subspace $\PP$ of $\R^n$.
 We will generally write $P$ to denote a $j$-plane in $\Omega$ for some subset $\Omega\subset \R^n$,
and $\PP$ to denote a complete $j$-plane.
\end{definition}

\begin{definition} \label{k-foliated} 
Let $n\in \N$, $n>1$, $\Omega$ be an open subset of $\R^n$.
We say a mapping $w\in C^0( \Omega, \R^\ell)$ is $(n-k)$-flatly 
foliated whenever  $0\le k <  n$ is an integer and there exists disjoint subsets  
$F_j, j=0,\ldots, k$ of $\Omega$, such that  the following properties
hold:
\begin{itemize}
\item[(i)] $\ds \Omega = \bigcup_{j=0}^k F_j$,
\item[(ii)] For all $j\in \{0, \cdots, k\}$, $\ds \Omega_j := \bigcup_{m=0}^j F_m$ is open,
\item[(iii)]  For all $j\in \{0, \cdots, k\}$ and every $x\in F_j$, there exists at least one $(n-j)$-plane $P$  in $\Omega_j$
such that $x\in P$ and $w$ is constant on $P$.  
\end{itemize} 
We say a mapping is flatly foliated when it is $(n-k)$-foliated for some integer 
$k$.
\end{definition}

\begin{remark}
Note that a straightforward conclusion of the above definition is that $F_j$ is closed in $\Omega_j$ for all $j \in \{0, \cdots, k\}$.
\end{remark}

\begin{definition}\label{defi-develop}
Let $n,N \in \N$, $n>1$, $N\ge 1$, and let $\Omega$ be an open subset of $\R^n$. We say a mapping $y\in C^1(\Omega, \R^N)$ is $(n-k)$-developable whenever 
$Dy: \Omega\to \R^{N\times n} \cong \R^{nN} $ is $(n-k)$-flatly foliated. We say a 
mapping is developable when it is $(n-k)$-developable for an integer $k \in \{0,1, \ldots, n-1\}$.
%
%
\end{definition}  We will later introduce weaker versions of the notions defined in Definitions \ref{k-foliated} and \ref{defi-develop} for 
mappings which are not necessarily of the required regularity. 

The following two theorems  sum up the main contribution 
of this paper. The first theorem concerns Sobolev isometric immersions of 
Euclidean domains and extends Theorem \ref{0deve-regu} to arbitrary codimension.

\begin{theorem}
Let $k\in \{1,\ldots, n-1\}$. Assume that $\Omega$ is a bounded, open subset of $\R^n$, and that 
$U\in W^{2,p}_{loc}(\Omega;\R^{n+k})$ is  an isometric immersion, so that
$U$ satisfies
\[
U_{x^i}\cdot U_{x^j}  = \delta_{ij} \quad \mbox{a.e. in}  \,\, \Omega, \quad \forall i,j\in \{1,\ldots, n\}.
\]  If $p\ge\min \{ 2k, n\}$ then $U\in C^{1}(\Omega;\R^{n+k})$,
and $U$ is $(n-k)$-developable.
\label{T.isom_immersion}\end{theorem}

The next theorem is a similar statement concerning scalar functions and 
generalizes  to arbitrary $n$ and $k$  those parts of Theorem  \ref{MArigid}
that concern the (pointwise) degnerate Monge-Amp\`ere equation \eqref{MA-2d} .
This result is new whenever $n>2$, even for $k=1$. 

\begin{theorem}
Assume that $\Omega$ is a bounded, open subset of $\R^n$
and that $u:\Omega\to \R$ satisfies
\begin{equation}
u\in W^{2,p}_{loc}(\Omega)\mbox{ with } p \ge \min \{2k, n\} \quad\quad \mbox{rank}(D^2u)\le k \  a.\,e
\label{w2pstar}\end{equation}
for some $k\in \{1,\ldots, n-1\}$. Then $u\in C^1(\Omega)$,
and  $u$ is $(n-k)$-developable.
\label{T.degenerateMA}\end{theorem}

\begin{remark} 
One interesting feature of these results is that the Sobolev regularity $W^{2,p}$ can be much below the
required $W^{2,n+\varepsilon}$ for obtaining $C^1$ regularity by Sobolev embedding theorems.  The  
argument used in  \cite[Lemma 2.1]{Pak} to show 
the continuity of the derivatives of the given Sobolev isometry is 
no more generalizable to our case. In \cite{Pak}, the $C^1$ regularity is shown as a  first step towards  the proof of developability. 
Here, on the other hand, we first show a weaker version of developability for the mapping  and use it to show the $C^1$ regularity. 
\end{remark}

\begin{remark}
In Example \ref{ex2} below, we show that if $u\in W^{2,p}(\Omega)$
satisfies $\mbox{rank}(D^2u)\le k \  a.\,e$, and if $p<k+1$, then $u$ may fail
to be $C^1$. These examples in particular imply that the condition $p\ge\min\{ 2k, n\}$
in Theorem \ref{T.degenerateMA}
cannot be weakened if $k=1$ or $k=n-1$. We believe however that it can be
weakened if $k\in\{2,\ldots, n-2\}$. Indeed, it seems likely that the conclusions of the theorem continue to hold  under the assumption that 
\beq
u\in W^{2,p}_{loc}(\Omega)\mbox{ with } p \ge k+1 \quad\quad \mbox{rank}(D^2u)\le k \  a.\,e.
\label{conjjjj}\eeq
\end{remark}

\subsection{Some examples}



\begin{example}
For any 
$k<n$ and $1\le p<k+1$, there exists $u\in W^{2,p}_{loc}(\R^n)$,
and  
rank$\,(D^2 u) \le k$ a.e.,  but such that the conclusions of the theorem
fail. Indeed, consider $u$ of the form
\[
u(x^1,\ldots, x^n) = u_0(x^1,\ldots, x^{k+1})
\quad\mbox{ for $u_0\in C^2_{loc}(\R^{k+1}\setminus \{0\})$ 
homogeneous
of degree $1$.}
\]
One easily checks that $u\in \cup_{p<k+1}W^{2,p}_{loc}(\R^n)$, and it is clear that
$Du$ is not continuous on the set $\{ x\in \R^n : x^1 = \ldots, x^{k+1} = 0\}$, unless it is constant.


\label{ex2}\end{example}

One could generalize the above example by gluing conic singularities in the following manner:  

\begin{example}
 
By Vitali's covering theorem, 
we choose a covering  ${\mathcal B}:= \{B(a_i, r_i)\}_{i\in \mathbb N}$ of $\R^{k+1}$ of non-overlapping  balls
so that $\R^{k+1} \setminus \bigcup_{i\in \mathbb N} B(a_i, r_i)$ is of Lebesgue measure zero.  We define  $v_0: \R^{k+1} \to \R^{k+1}$   by 
\begin{equation*} 
\ds v_0(x):=\left  \{ \begin{array}{ll} a_i+ r_i (x-a_i)/|x-a_i| & \mbox{if} \,\, x\in B(a_i, r_i), \\ x & \mbox{otherwise.}   \end{array} \right. 
\end{equation*}   It can be easily verified that $v_0 \in W^{1,p}_{loc}(\R^{k+1})$  for all $1 \le p<k+1$ and that $v_0=Du_0$ for a scalar function. 
Let $u(x^1, \ldots, x^n) := u_0(x^1, \ldots, x^{k+1})$. Then $u\in W^{2,p}_{loc}(\R^n)$  for $1 \le p < k+1$,  
rank$\,(D^2 u) \le k$, but $Du$ is not continuous on the set $\{a_i\}_{i\in \mathbb N} \times \R^{n-k-1}$. 

\label{ex-bes}\end{example}  

One might naively hope that for every $k<n$, the set $\{x\in \Omega : \mbox{rank}(D^2 u) = k\}$
is foliated by $n-k$-planes on which $Du$ is constant. This is not at all the case.
 
\begin{example}
Consider $u:(0,1)^2\to \R$ of the form $u(x,y) = F(x)$ where $F' = f:(0,1)\to \R$ is a {\em strictly increasing} Lipschitz continuous function such that $\{x\in (0,1): f'(x)=0\}$ has positive measure.
For example, fix an open dense set
 $O\subset (0,1)$ whose complement has positive measure,
and let 
$f(x) := \calL^1( (0,x) \cap O)$,
so that $f$ is Lipschitz continuous and
\[
f'(x) =\begin{cases}1 &\mbox{ for }a.e. \ x\in O\\
0&\mbox{ for }a.e. \ x\not\in O.
\end{cases}
\]

For a function of this form, we have $u\in W^{2,\infty}$, with
\[
Du(x,y) = (f(x), 0), \quad \quad D^2 u(x,y) = \left(\begin{array}{cc}
f'(x) &0\\0&0\end{array}\right)\ \  a.e. 
\]
so that  rank$(D^2u) \le 1$ a.e., and rank$(D^2 u) = 0$ on a dense set of positive measure.
However, 
there is no $2$-dimensional set on which $Du$ is locally constant; rather, for every $\xi\in \mbox{Im}(Du)$,
where $\mbox{Im}(\cdot)$ denotes the image,
$Du^{-1}\{ \xi \}$ is the line segment  $f^{-1}\{ \xi\} \times (0,1)$.
\label{ex3}\end{example}

\begin{example}
Consider again $u:(0,1)^2\to \R$ of the form $u(x,y) = F(x)$, where $F'  = f$
and
$f(x) := \calL^1( (0,x)\setminus O)$, where $O$ is as in Example \ref{ex3} above.
Then $f$ is Lipschitz continuous and
\[
f'(x) =\begin{cases}0 &\mbox{ for }a.e. \ x\in O\\
1&\mbox{ for }a.e. \ x\not\in O.
\end{cases}
\]

Then in the notation of Definition \ref{k-foliated} below,
$\Omega = \Omega_1$, and $\Omega_0 = O \times (0,1)$.
Thus $\Omega_0$ is a dense subset of $\Omega_1$, and $F_1 =
\Omega_1\setminus \Omega_0$ is nowhere dense in $\Omega_1$.

More generally, given $0 \le j < k \le n$, one can write down  examples in
the same spirit defined on the unit cube in $\R^n$, such that 
$\Omega_j$ is dense in $\Omega_k$.
\label{ex.dense}\end{example}

\begin{example}\label{ex:4}
Fix a $C^2$ map $\vec v: \R\to \R^2$ such that $\vec v(0) = 0$, $v'(z)\ne 0$ for $z\ne 0$, 
and $\lim_{z\to 0}\frac{\vec v'}{|\vec v'|}$ does not exist.
For example, we may take $\vec v(z) = (z^5 \cos(1/z), z^5\sin (1/z))$.

Now set $\Omega = (-1,1)^3$, and let $u(x,y,z) = (x,y)\cdot \vec v(z)$.
Then we can write $Du(x,y,z) = (\vec v(z), (x,y)\cdot \vec  v\mbox{\,}'(z))$.
Thus level sets of $Du$ are the plane $z=0$, together with the line segments 
\[
\{ x,y,z) : z = z_0,
(x,y)\cdot  \vec v\mbox{\,}'(z_0) = c \}, \qquad z_0\ne 0, c\in \R.
\]
It is also easy to check
that $u$ is $C^2$,  rank$(D^2 u) = 2$ if $z\ne 0$ and  rank$(D^2 u) = 0$ if $z=0$.

(Note also,  $\tilde u := u + z^2$ has all the same properties as $u$ described above, except that
rank$(D^2 u)=1$ when $z=0$.)

This example show that (in notation to be introduced later)
$\bar \Omega^k$ may contain planes of dimension greater than $n-k$
on which $Du$ is a.e. constant. By contrast, the previous example shows that it may
also happen that $\bar \Omega^k\setminus \Omega^k$ is foliated
by planes of dimension $n-k$. 

Also, we can see from this example that  
the $(n-k)$-planes that locally foliate $\Omega^k$ may oscillate wildly
as one approaches points in $\bar \Omega^k$ at which
rank$(D^2 u) < k$.

\end{example}

\medskip 

\subsection{Remarks on notation, and an outline of proofs}
\label{background}

Throughout the paper, we will often simply write ``measurable", ``almost everywhere", without specifying the Hausdorff measure at use, when 
the latter is clear from the context.  Many of our arguments take place in a product space $\Omega\times \R^\ell$,
where $\Omega\subset \R^n$ and $\ell$ is a positive integer. 
In this setting we will think of $\Omega$ and $\R^\ell$ as ``horizontal" and ``vertical",
respectively, and we will use subscripts $h$ and $v$  accordingly.
For example, we will write $p_h, p_v$ to designate projections
of $\Omega\times \R^\ell$ onto the horizontal and vertical factors, respectively:
\begin{equation}
p_h(x,\xi) := x,
\qquad \qquad
\quad p_v(x,\xi) := \xi.
\label{phv.def}\end{equation}
\

If $w\in L^p(\Omega)$ for some $p<\infty$, then a Lebesgue point of $w$ 
will mean a point $x$
such that
\begin{equation}\label{lebpt}
\lim_{r\to 0} \barint_{B_r(x)}|w(y) - w(x)|^p \ dy \ := \ 
\lim_{r\to 0} \frac 1{\calL^n(B_r(x))} \int_{B_r(x)}|w(y) - w(x)|^p \ dy  \ = \ 0.
\end{equation}
Thus, we always understand ``Lebesgue point" in an $L^p$ sense.
We  assume that every function $w$ appearing in this paper is
precisely represented. Thus $w$ always equals its Lebesgue value
at every point where the Lebesgue value exists. If $u\in W^{2,p}(\Omega)$, 
there is a  set $E$ such that $\mbox{Cap}_p(E) = 0$ 
and every point of $\Omega\setminus E$
is a Lebesgue point of $Du$. 
The capacity estimate implies that $\calH^{n-p+\e}(E)=0$ for every $\e>0$.
These facts can be found for example in
Ziemer \cite{Ziemer}, Theorem 3.3.3 and 2.6.16 respectively, or in \cite{eg}.


\medskip

To describe the proof, it is useful to introduce several weaker versions of the the notions
of flatly foliated, defined above.

\begin{definition} \label{dWk-foliated} 
Let $n\in \N$, $n>1$, $\Omega$ be an open subset of $\R^n$.
We say a measurable mapping $w : \Omega \to \R^\ell$ is densely weakly $(n-k)$-flatly 
foliated whenever  there exist some $k\in \{0, 1,\ldots, n-1\}$ and 
disjoint subsets $F_j, j=0,\ldots, k$ of $\Omega$, such that
\begin{equation}
\Omega = \bigcup_{j=0}^k F_j,
\label{Fj1.2}\end{equation}
and in addition,   the following properties
hold for every $j$:
\begin{equation}
\quad\quad
\Omega_j := \bigcup_{m=0}^j F_m \mbox{ is open,}
\label{Fj2.2}\end{equation} and 
\begin{multline} 
\mbox{ for every $x$ in some dense subset of $F_j$, there exists at least one $n-j$-plane $P$  in $\Omega_j$} \\
\mbox{such that $x\in P$ and 
$w$ is $\calH^{n-j}$ a.e. constant on $P$.}
\label{Fj3.2}\end{multline}
\end{definition}

\begin{definition}\label{Wk-foliated}
Let $n\in \N$, $n>1$, $\Omega$ be an open subset of $\R^n$.
We say a measurable mapping $w : \Omega \to \R^\ell$ is pointwise weakly $(n-k)$-flatly 
foliated whenever  there exist some $k\in \{0, 1,\ldots, n-1\}$ and 
disjoint subsets $F_j, j=0,\ldots, k$ of $\Omega$, such that 
\eqref{Fj1.2} and \eqref{Fj2.2} hold, and 
\begin{multline} 
\mbox{ for every $x\in F_j$, there exists at least one $n-j$-plane $P$  in $\Omega_j$} \\
\mbox{such that $x\in P$ and 
$w$ is $\calH^{n-j}$ a.e. constant on $P$.}
\label{Fj3.2a}\end{multline}
\end{definition}

\begin{remark} 
The definitions require that the values of $w$ are well defined for $\calH^{n-j}$ a.e. points on the given $n-j$-planes in $\Omega$. 
As noted above, this is the case if we assume that e.g. $w\in W^{1,k+1}_{loc}(\Omega, \R^\ell)$
and $w$ is precisely represented, since in that 
case the set of points that fail to be Lebesgue points of $w$ has dimension less than $n-k$. 
\end{remark}

We start in Section \ref{S:6} by showing that if $U\in W^{2,2}(\Omega;\R^{n+k})$
is an isometric immersion for $\Omega\subset\R^n$, then $w = DU$ satsfies
\[
\mbox{rank}(Dw) \le k \ \ a.e. \mbox{ in }\Omega.
\]
This is a classical fact for smooth maps. As a consequence, both of our main results
reduce to the study of maps 
$w:\Omega\to \R^\ell$ for some $\ell$, such that 
\begin{equation}
\mbox{rank}(Dw(x)) \le k \ \ a.e. \mbox{ in }\Omega, 
\qquad
w = (Du^1,\ldots, Du^q)\mbox{ for some }q\ge 1.
\label{unified}\end{equation}

A main challenge we must address is to find a way
to extract information from the hypotheses \eqref{unified}
under conditions of low regularity.
We carry this out making extensive use of the machinery of geometric measure theory, including in particular some results from Giaquinta, Modica and Sou\v cek \cite{gms},
Fu \cite{fu1} and the first author  \cite{J2010}
about the related topics of Cartesian maps and Monge-Amp\`ere functions.

To explain the role of geometric measure theory, we first outline the basic argument on a formal level.
Toward that end, consider a {\em smooth} map $w = (Du^1,\ldots, Du^q)$ such 
that rank$(Dw)= k$ everywhere, and further suppose that
\begin{itemize}
\item Image$(w)$ is a smooth embedded $k$-dimensional submanifold $\Gamma_v\subset \R^n$,
where  Im$(w)$ denotes the image of $w$, and  
\item For every $\xi\in \Gamma_v$, $\Gamma_h(\xi) := w^{-1}\{\xi\}$ is a smooth $(n-k)$-dimensional
submanifold of $\Omega$.
\end{itemize}
These assumptions are far stronger than one can reasonably expect, but in any case they are certainy consistent with
the condition that rank$(Dw) = k$.
For every $\xi\in \Gamma_v$, and for every $x\in \Gamma_h(\xi)$, basic calculus implies that
\begin{equation}
\mbox{Im}(Dw(x)) = T_\xi \Gamma_v
\label{TxiGv}\end{equation}
and
\begin{equation}
\ker(Dw(x)) = T_x\Gamma_h(\xi).
\label{TxGh}\end{equation}
Moreover, the symmetry of $D^2 u^i(x)$
implies that $\ker(D^2 u^i(x)) = [ \mbox{Im}(D^2 u^i(x))]^\perp$, if we identify,
in the natural
way,
the horizontal and vertical spaces to which $T_\xi\Gamma_v$ and $T_x\Gamma_h(\xi)$ belong.
Thus
\begin{align*}
T_x\Gamma_h(\xi) = \ker(D w(x)) = \cap_{i=1}^{q} \ker(D^2 u^i(x)) 
= 
\cap_{i=1}^{q}  [ \mbox{Im}(D^2 u^i(x))]^\perp.
\end{align*}
The space on the right is completely determined by $T_\xi \Gamma_v$ --- in
fact it can be written $\cap_{i=1}^{q}    [ P_i T_\xi \Gamma_v]^\perp$,
where $P_i$ denotes orthonormal projection of $\R^{n q} = (\R^n)^q$ onto 
the $i$th copy of $\R^n$.
Thus  the tangent space $T_x\Gamma_h(\xi) $ does not depend 
at all on $x\in \Gamma_h(\xi)$, but only on $\xi$.
Since the tangent space is constant, $\Gamma_h(\xi)$ must be a union of $n-k$-planes
in $\Omega$, all orthogonal to $\cap_{i=1}^{j}    [ P_i T_\xi \Gamma_v]^\perp$.

The rigorous version of this argument starts in  Section \ref{s:dcm}, where we use the machinery of 
geometric measure theory to establish facts about
\begin{itemize}
\item
 the structure of $\Gamma_v$ and $\Gamma_h(\xi)$, which in our actual
 proof will be, not exactly the image and the level sets of $w$,
 but closely related sets;
and 
\item the relationship between their tangent spaces and the derivatives of $w$,
along the lines of \eqref{TxiGv} and \eqref{TxGh} above
\end{itemize}
that are (barely)
strong enough to justify some form of the proof sketched above.
These arguments apply to general mappings  (without a gradient structure)
$w\in W^{1,k+1}(\Omega ; \R^\ell)$
such that rank$(Dw)\le k$ a.e.
Under these assumptions, we obtain $\Gamma_v$ and $\Gamma_h(\xi)$
as, essentially, the vertical projection and horizontal slices, respectively,
of a set
\[
\Gamma := \{ (x, w(x))\in \Omega\times \R^\ell : x\mbox{ is a Lebesgue point of both $w$ and $Dw$}\}.
\]
(See \eqref{Gammav.def}, \eqref{Gammah.def} for the actual definitions.)
Appealing to results of Giaquinta, Modica and Sou\v cek \cite{gms}, we find
that  that $\Gamma$ is $n$-rectifiable, and that an integral $n$-current $G_w$,
canonically associated to the graph of $w$ and carried by $\Gamma$,
has no boundary in $\Omega\times \R^\ell$. Then the rectifiability of $\Gamma_v$ and
of $\calH^k$ almost every $\Gamma_h(\xi)$ follows from  classical results and
the definitions of these sets, as does
a version of \eqref{TxiGv}. Additional work is required to establish a version of \eqref{TxGh}
and to show that the slices $\Gamma_h(\xi)$ have enough regularity (in particular,
they carry integer $n-k$-currents with no boundary) to conclude
from the constancy of the tangent spaces that they are in fact planar.

In Secion \ref{fullrank}, we use these facts to prove that if $w\in W^{1,k+1}_{loc}$
satisfies \eqref{unified}, then $w$ is densely weakly $(n-k)$ flatly foliated.
More precisely, we define
\begin{equation*}
\Omega^k := \{ x\in \Omega : \mbox{$x$ is a Lebesgue point of $w$ and $Dw$, and
$\mbox{rank}\,(D w) = k $} \},
\end{equation*} 
and we give a rigorous version of the formal argument sketched above to
show, roughly speaking, that $\Omega^k$ is almost everywhere foliated by level sets of $w$
that are $n-k$-planes in $\Omega$. (We remark that this is the {\em only}
place in the paper where we use the gradient structure of $w$.)
To deduce that $w$ is  densely weakly $(n-k)$-flatly foliated, we define $F_k := \bar \Omega^k$
and  $\Omega_{k-1}:= \Omega \setminus F_k$, and we note that 
rank($D^2 u$)$\le k-1$ a.e. in $\Omega_{k-1}$. Hence the above machinery could be re-applied to the new set with the new rank condition. 
More generally, letting $\Omega_k = \Omega$, and for $j\in \{k, \ldots, 0\}$, defining 
(working downwards) 
\begin{align*}
\Omega^{j} 
&:=  \{ x\in  \Omega_{j} :
 \mbox{$x$ is a Lebesgue point of $Du$ and $D^2u$, and
$\mbox{rank}\,(D^2 u) = j $} \}, \\  
F_j & := \bar \Omega^j \cap \Omega_j \, ,   \\
\Omega_{j-1}  & := \Omega_j - F_{j}  = \Omega_j - \bar \Omega^{j} , 
\end{align*}
we  obtain a partition of $\Omega$ into disjoint sets $F_j$, $j=0,1,\ldots, k$
such that every $F_j$, has a dense subset
foliated by $n-j$-planes on which $w$ is $\calH^{n-j}$ a.e. constant.

\medskip

Following this, we prove in Section \ref{S:3} that
if $w\in W^{1, k+1}_{loc}(\Omega; \R^\ell)$ 
is  {\em densely} weakly  $(n-k)$-flatly foliated, then $w$ is
{\em pointwise} weakly  $(n-k)$-flatly foliated. (In fact here
we only need $W^{1,p}_{loc}$ for some $p>k$.) 
The hypothesis already yields a partition of $\Omega$
into sets $F_j$ satisfying properites  \eqref{Fj1.2}, \eqref{Fj2.2},
and so the point is to show that \eqref{Fj3.2} together with
the assumed Sobolev regularity implies \eqref{Fj3.2a}.
To do this, we obtain a planar level set of $w$ through a 
given point as a limit of planar level sets through nearby points.
We remark that it is possible, as illustrated in Example \ref{ex3}, for $F_k$
to contain a subset of
$\Omega\setminus\Omega^k$ of positive
measure to be foliated by $n-k$-planes on which $w$ is constant.

The arguments of Sections \ref{s:dcm}, \ref{fullrank} and \ref{S:3} require only the weaker
regularity assumption \eqref{conjjjj}, and this 
hypothesis is sharp in a sense; this follows from Example
\ref{ex2} below.  The stronger assumption \eqref{w2pstar}
is needed for Section \ref{S:4}, in which 
prove  that if $p=\min\{2k,n\}$ and
$w\in W^{1, p}_{loc}(\Omega; \R^\ell)$ 
is  {\em pointwise} weakly  $(n-k)$-flatly foliated, then $w$ is
continuous, and hence $(n-k)$-flatly foliated. 
This will complete the proof of 
our main results.
For the proof, we first show show that if a point $x\in F_k$
is contained in two distinct $n-k$-planes in $\Omega$ on which $w$ is 
{\em a.e.}
constant, then the two constants are in fact
equal. (Example \ref{ex:4}  shows that this situation can in fact arise.)
It follows rather easily from this that the restriction of $w$ to $F_k$
is $C^0$, and indeed that the same holds in $F_j$
for all $j\le k$. To conclude that $w$ is continuous in $\Omega$,
it remains to show that it is continuous at points of
$\partial \Omega_j\cap \Omega$. This is a little more subtle,
and is proved by showing that any such discontinuity is
inconsistent with the $p$-quasicontinuity of $w$, given
facts we have already established about $w$.

The condition $p\ge \{2k,n\}$ is sharp for the results of Section
\ref{S:4}, at least for certain values of $k$,
including $k=2,4,8$. This follows from Examples \ref{ex:Hopf1} - 
\ref{ex:Hopf3} in Section \ref{S:4}.
These results however apply to vector-valued  maps  $w:\Omega\to \R^\ell$
that are 
pointwise a.e. flatly foliated. As suggested above, we believe
that if one considers maps that in addition possess a gradient
structure, that is, maps of the form $w = (Du^1,\ldots, Du^q)$ for
some $q$, then it should be possible to 
weaken the regularity requirements.

 \bigskip

\noindent {\bf Acknowledgments.}
The first author was partially supported by the National Science and
Engineering Research Council of Canada under operating grant 261955.  
The work performed on the project by the second author was partially supported by 
the NSF grant  DMS-1210258.

\section{Degeneate Hessians for Sobolev isometric immersions}\label{S:6}

In this section we  prove a proposition that reduces the case of
isometries to that of maps whose Hessian satisfies
a degeneracy condition.  This is a variant of a classical lemma
of  Cartan \cite{Cart}, which concerns smooth maps and has a correspondingly stronger conclusion.

\begin{proposition}\label{isom.-is-degenerate}
Assume that $\Omega \subset \R^n$ is a bounded open set, and that $U\in W^{2,2}(\Omega, \R^{n+k})$ is an isometric 
immersion of $\Omega$ into $\R^{n+k}$ for some 
$k\in \{1, \cdots, n-1\}$, i.e. $U$ satisfies 
\begin{equation}\label{isometry} 
U_{x^i} \cdot U_{x^j} =\delta_{ij}, \quad \forall i,j \in \{1,\cdots, n\}. 
\end{equation}  
Let $w := DU : \Omega\to \R^n\otimes \R^{n+k}  \cong \R^\ell$ for $\ell = n(n+k)$.
Then 
$$
\mbox{rank}(D w) \le k \quad a.e. \,\, in \,\,  \Omega. 
$$
 \end{proposition} 

In the proof of this result only, to simplify notation we will write  $U_{,i}$ to denote
partial differentiation with respect to the $i$th coordinate direction.

\begin{proof}

We will first establish the following identity:

\begin{equation}\label{flatness}
\ds U_{,ij}\cdot U_{,kl} - U_{,il}\cdot U_{,jk} =0 \quad \forall i,j,k,l \in \{1,\cdots, n\} \quad \mbox{a.e. in} \,\, \Omega.
\end{equation} Let $U_m\in C^\infty(\Omega, \R^{n+k})$  be a sequence of mappings converging to $U$ in the $W^{2,2}$-norm, and let 
$g^m_{ij}:= U_{m,i} \cdot U_{m,j}$. Twice differentiating $g^m_{ij}$ we obtain for all $i,j,k,l$:
$$
\ds g^m_{ij,kl} = U_{m,ikl} \cdot U_{m,j} + U_{m,ik} \cdot U_{m,jl} + U_{m,il} \cdot U_{m,jk} + U_{m,i} \cdot U_{m,jkl}.
$$ Permuting the indices and canceling the terms in third derivatives yields:
$$
\ds g^m_{ij,kl} + g^m_{kl,ij} - g^m_{il,jk} - g^m_{jk,il} = -2 (U_{m,ij} \cdot U_{m,kl} - U_{m,il} \cdot U_{m,jk}).
$$ Passing to the limit as $m\to \infty$, we observe that the left hand side converges in the sense of distributions to 0, 
while the right side converges in $L^1$ to  $-2(U_{,ij}\cdot U_{,kl} - U_{,il}\cdot U_{,jk})$. This establishes (\ref{flatness}). 
Our second observation is that
\begin{equation}\label{orthog} 
\ds U_{,ij} \cdot U_{,k} = 0 \quad \forall i,j,k\in \{1,\cdots, n\} \quad \mbox{a.e. in}\,\, \Omega.
\end{equation} This is straightforward to see, as differentiating the isometry constraint (\ref{isometry}) we obtain for all $i,j,k$:
$$
\ds 0= U_{,ik} \cdot U_{,j} + U_{,i} \cdot U_{,jk} =U_{,ij}\cdot U_{,k} + U_{,i}\cdot U_{,kj}= U_{,ki} \cdot U_{,j} + U_{,k} \cdot U_{,ji} ,
$$  where the two last identities are obtained by permutations in $i,j,k$ and all three are valid a.e. in $\Omega$. Now, adding the first two identities and 
subtracting the third implies (\ref{orthog}), considering that $U_{,ij} = U_{,ji}$ for all choices of $i,j$ a.e. in $\Omega$. 

In order to proceed, for any $x\in \Omega$ for which the identities (\ref{isometry}), (\ref{flatness}) and (\ref{orthog}) are valid- hence for a.e. $x\in\Omega$ -, 
we define the orthogonal space to the image $U(\Omega)$ at the point $U(x)$ to be:
$$
O(x):= \mbox{span} <U_{,1}(x), \cdots, U_{,n}(x)>^\perp,  
$$ and the symmetric bilinear form ${\mathcal B}(x) : \R^n \times \R^n \to O(x)$ by 
$$
\ds {\mathcal B}(x) (V, W)= W \cdot D^2 U(x) V:= \sum_{m=1}^{n+k} (W\cdot D^2 U^m(x) V) \vec {e}_m,
$$
where  $U= (U^1, \cdots, U^{n+k})$. 
Evidently (\ref{orthog}) implies that $\mathcal B(x)$ takes values in $O(x)$. On the other hand (\ref{flatness}) implies that for all $X,W,Y,Z\in \R^n$ we have
$$
\ds {\mathcal B}(x)(X, W)\cdot  {\mathcal B}(x) (Y,Z) -  {\mathcal B} (x) (X,Z)  \cdot {\mathcal B}(x) (Y,W) =0,
$$ i.e. the symmetric bilinear form $ {\mathcal B}(x)$ is flat with respect to the Euclidean scalar product on $O(x)$. Hence, we can apply a  
result due to E. Cartan \cite{Cart} (See also \cite[Lemma 1]{VWKG} for a proof), to obtain that 
$$
\ds \mbox{dim}(\ker {\mathcal B(x)}) \ge \mbox{dim} (\R^n) - \mbox{dim}(O(x)) = n-k,
$$ where
$$
\mbox{ker} \,  ({\mathcal B(x)}  ):= \{ V\in \R^n; {\mathcal B(x)} (V,W) =0 \,\, \forall W\in \R^n \} =  \ker(Dw(x)). 
$$
 This completes the proof of the proposition. 
\end{proof}

\section{Degenerate Cartesian maps}\label{s:dcm}
 
In this section, $\Omega$  is as usual a bounded, open subset of $\R^n$, 
and $w$ is a map satisfying
\beq
w\in W^{1,k+1}_{loc}(\Omega, \R^\ell) \, ,
\quad\quad
\mbox{rank}\,(D w) \le k \ \ \mbox{a.e.}
\label{ww1kplus1}\eeq
for some $k\in \{ 1,\ldots, n-1\}$ and some $\ell \ge 1$. 
We will use the notation
\begin{align}
\Lambda_w 
&:= 
\{ x\in \Omega : x\mbox{ is a Lebesgue point of both $w$ and $Dw$} \}
\label{Lambdaw.def}\
\\
\Gamma 
&:= 
\{ (x, w(x)) : x\in \Lambda_w \} \subset \Omega\times \R^\ell
\label{Gamma.def}\\
\Gamma_h(\xi) &:= \{ x\in \Lambda_w :  w(x) = \xi \}     \label{Gammah.def}  \\
\Gamma_v  &:= \{ \xi \in \R^\ell \ : \ \calH^{n-k}( \Gamma_h(\xi) )>0 \} \label{Gammav.def} \\
 \Omega^k & = \{ x\in \Lambda_w : \mbox{rank}(Dw(x)) = k \} . \label{Omegak.def}
\end{align}
The main result of this section, stated below, will be used to make
precise the formal arguments discussed in Section \ref{background}.
Terminology appearing in the proposition
will be recalled after its statement.

\begin{proposition}
Assume that $w$ satisfies \eqref{ww1kplus1}. Then 
$\Gamma_v$ is $k$-rectifiable, and for $\calH^k$ a.e. $\xi\in \Gamma_v$, 
the following hold:
\begin{equation}
\mbox{ 
$\Gamma_h(\xi)$ is $\calH^{n-k}$-measurable and $n-k$-rectifiable}
\label{Gvs1}\end{equation}
\begin{equation}
T_\xi\Gamma_v = \mbox{Im}(Dw(x))
\  \mbox{ and } \ 
\ker(Dw(x)) = T_x\Gamma_h(\xi),
\ \ \quad\mbox{ $\calH^{n-k}$ a.e. in $\Gamma_h(\xi)$.}
\label{Jpos_ae}\end{equation}
In addition, for $\calH^k$ a.e. $\xi \in \Gamma_v$, there exists an integral
current $H_\xi$ in $\Omega\times \R^\ell$, defined explicitly
in \eqref{Hxi.def} below,
represented by integration over $\{ \xi\} \times \Gamma_h(\xi)$ such that
$\partial H_\xi = 0$. 
Finally, 
\begin{equation}
\calL^n\Big( \Omega^k   \setminus \cup_{\xi \in \Gamma_v^*} \Gamma_h(\xi) \Big) = 0,
\label{weakfoliate}\end{equation}
where
\begin{equation}
\Gamma_v^* := \{ \xi \in \Gamma_v : \partial H_\xi=0, \mbox{ and 
\eqref{Gvs1} and  \eqref{Jpos_ae} hold.} \}.
\label{Gvstar.def}\end{equation}
\label{P.weakfol}
\end{proposition}

This is related to results in \cite{J2010}, proved in the more
abstract setting of Monge-Amp\`ere functions. 
Here, we are able to exploit
the Sobolev regularity and results of Giaquinta {\em  et al} \cite{gms}
to extract more information than in \cite{J2010},
such as conclusions \eqref{Jpos_ae}, which are new. 
We also believe that the arguments given here are more 
transparent than those of \cite{J2010}.

\begin{remark} 
We emphasize that $\Gamma$ and  $\Gamma_v$  may differ from the  
graph $\{ (x,w(x)) : x\in \Omega\}$ and the image $w(\Omega)$
by sets of positive $\calH^n$ measure.
Indeed, \cite{MalyMartio} establishes the existence of a continuous mapping $w\in W^{1,n}(\Omega;\R^n)$ with vanishing Jacobian $($i.e. $k=n-1)$, 
for which $w(\Omega)$ has positive measure. In this construction, the bulk of the image is obtained by  applying $w$ to the null set $\Omega\setminus \Lambda_w$, and in fact
Proposition \ref{P.weakfol} shows that  $\Gamma_v$ is an $n-1$-rectifiable set.
\end{remark} 

We start by recalling some definitions.
For more background, one can consult for example
\cite{gms} for
a general introduction to geometric measure theory in product spaces 
and whose notation we have tried to follow.

If $U\subset \R^L$ for some $L$, then we  say that $\Gamma\subset U$ is $j$-rectifiable if 
\[
\Gamma\subset M_0 \cup \bigcup_{q=1}^\infty f_q(\R^j),\qquad\mbox{ where }\calH^j(M_0)=0\mbox{ and }
f_q:\R^j\to U \  \mbox{ is Lipschitz}.
\]
It is a standard fact that a $j$-rectifiable set $\Gamma$ has a $j$-dimensional
approximate tangent plane, denoted $T_y\Gamma$, at $\calH^j$ almost every
$y\in \Gamma$.

If $\PP$ is a $j$-dimensional plane in some $\R^L$, then a {\em unit $j$-vector orienting $\PP$}
is a $j$-vector (that is, an element of the space $\Lambda_j \R^L$) of the form $\tau = \tau_1\wedge \cdots \wedge \tau_j$, where $\{\tau_i\}_{i-1}^j$
form an orthonormal basis for the tangent space to $\PP$.

Let $\calD^j(U)$ denote the space of smooth, compactly supported $j$-forms on $U$.

Heuristically, $j$-currents supported in $U$ are ``generalized submanifolds" of dimension $j$,
defined by duality to $\calD^j(U)$. Integer multiplicity (henceforth abbreviated as {\em i.m.}) rectifiable currents  are those which are represented by a superposition of rectifiable sets. 
More precisely,
an i.m. rectifiable $j$-current $T$ in $U$
is a bounded
linear functional on
$\calD^j(U)$
that may be represented in the form
\beq
T(\phi) \ = \ \int_\Gamma  \langle \phi, \tau\rangle \, \theta \, d\calH^n
\label{Tform}\eeq
where 
\begin{itemize}
\item$\Gamma$ is a $j$-rectifiable set,
\item$\theta:\Gamma\to {\mathbb N}$ is a $\calH^j$-measurable function,
locally integrable with respect to $\calH^j\rest \Gamma$;
and 
\item $\tau$ is a $\calH^j$-measurable function from $\Gamma$
into the space $\Lambda_j\R^L$ of $j$-vectors on $\R^L$, such that 
$\tau(y)$ is a unit $j$-vector that orients 
the approximate tangent space $T_y\Gamma$, for a.e. $y\in \Gamma$.
\end{itemize}
In \eqref{Tform}, we write $\langle \phi(y), \tau(y)\rangle$ to denote the dual pairing
between a $j$-covector $\phi(y)\in \Lambda^j\R^L$ and a $j$-vector $\tau(y)\in \Lambda_j\R^L$;
see \eqref{bracket.def} below for a concrete definition in the
product space setting.

When \eqref{Tform} holds,  we say that $T$ is represented by integration over $\Gamma$.

\medskip

We next introduce  notation needed to write these objects more explicitly, and in particular to write currents and differential forms in  the product space $U \ \Omega\times \R^\ell$.
 For $1\le j\le m$, we define
\beq
I(j,m) := \{ \alpha = (\alpha_1,\ldots, \alpha_j) : 1 \le \alpha_1 < \ldots <\alpha_j \le m\}.
\label{Ikm.def}\eeq
If $\alpha\in I(j,m)$ then $|\alpha| :=j$. 
We will think of $I(0,m)$ as consisting of a single element, ``the empty multiindex",
which we will denote $0$.

If $S = (S^i_j)$ is an $\ell\times n$ matrix (with $i$ running from $1$ to $\ell$ and $j$ from $1$ to $n$)
and  $\beta \in I(j,\ell), \gamma\in I(j,n)$ for some $j$
then 
\beq
S^\beta_\gamma = (S^{\beta_i}_{\gamma_{i'}} )_{i, {i'} = 1}^j \ , 
\quad\qquad
M^\beta_{\gamma}(S) := \det S^\beta_{\gamma}.
\label{minor.not}\eeq
We refer to $M^\beta_\gamma(S)$ as a {\em minor} of $S$ of order $j$.

We will  write points in $\Omega\times \R^\ell$ in the form $(x,\xi)$,
and we will write $\{ e_i \}_{i=1}^n$ and $\{ \e_j\}_{j=1}^\ell$
 to denote
the standard bases for the spaces 
\[
\R^n_h := \R^n\times \{0\} \ \ \ \ \
\ \ \ \mbox{ and }\ \ \ \ \ \R^\ell_v :=\{0 \}\times \R^\ell
\]
of ``horizontal'' and ``vertical" vectors.
For $\alpha\in I(j,n)$,  we set 
\[
dx^\alpha :=
dx^{\alpha_1}\wedge\ldots\wedge dx^{\alpha_j},
\qquad
e_\alpha := e_{\alpha_1}\wedge\ldots\wedge e_{\alpha_j}
\]
and similarly $d\xi^\beta$ and $e_\beta$, for $\beta \in I(j, \ell)$.
Thus, for example, every $n$-form in $\Omega\times \R^\ell$ may
be written 
\beq
\phi = \sum_{|\alpha|+ |\beta| = n} \phi_{\alpha\beta}(x,\xi) dx^\alpha\wedge d\xi^\beta,
\label{nform}\eeq
where it is understood that $\alpha\in I(*,n)$ and $\beta\in I(*,\ell)$. 
The dual pairing appearing in \eqref{Tform} is defined by
\begin{equation}
\langle  \sum_{|\alpha|+ |\beta| = n} \phi_{\alpha\beta} dx^\alpha\wedge d\xi^\beta,
\sum_{|\delta|+ |\gamma| = n} \tau^{\delta\gamma} \, e_\delta \wedge \ep_\gamma
\rangle \ = \   \sum_{|\alpha|+ |\beta| = n} \phi_{\alpha\beta} \tau^{\alpha\beta}.
\label{bracket.def}\end{equation}


Given  $\alpha\in I(j,n)$, we will write $\bar \alpha$ 
to denote the complementary multiindex, such that $(\alpha, \bar \alpha)$ is
a permutation of $(1,\ldots, n)$, and we write $\sigma(\alpha, \bar \alpha)$ to denote the
sign of this permutation. 
Hence $\bar \alpha$ and $\sigma(\alpha, \bar \alpha)$ are characterized by the conditions
\[
|\alpha| + |\bar \alpha|= n\qquad\mbox{ and }\qquad dx^\alpha\wedge dx^{\bar \alpha} =  \sigma(\alpha, \bar\alpha) dx^1\wedge \ldots\wedge dx^n.
\]
We then define the 
$n$-current $G_w$ by
\beq
G_w(\phi \ dx^\alpha \wedge d\xi^\beta) \ = \ 
 \sigma(\alpha,\bar\alpha)\int_{\Omega} \phi(x, w(x)) M^\beta_{\bar\alpha}(Dw) \ dx,
\label{Gw.def}\eeq
for $\phi\in C^\infty_c(\Omega\times \R^n)$ and  $|\alpha|+|\beta|=n$. 
(We use the convention that $M^0_0(Dw) = 1$.)

We will repeatedly use the fact that
\begin{equation}
G_w(\phi \, dx^\alpha\ \wedge d\xi^\beta) = 0 \qquad\mbox{ if }|\beta|\ge k+1,
\label{nulll}\end{equation}
which is a direct consequence of  \eqref{ww1kplus1}.
A computation (see \cite{gms}, section 3.2.1) shows that
\[
G_w(\phi) = \int_{\Lambda_w} W^* \phi,\	\qquad\mbox{ for every $n$-form $\phi$ in $\Omega\times \R^\ell$, where }W(x) := (x, w(x))
\]
and the pullback $W^*\phi$ is defined pointwise in $\Lambda_w$. 
Thus, $G_w$ formally looks like integration over the (oriented) graph of $w$; this is the motivation
for the definition of $G_w$. The next
lemma collects some useful observations of
Giaquinta, Modica and Sou\v cek \cite{gms} which clarify the sense in which this is,
and is not, the case.

\begin{lemma}
Assume that $w$ satisfies \eqref{ww1kplus1}.
Then: 
\begin{enumerate}

\item 
The restriction of $W(x) = (x, w(x))$ to $\Lambda_w$ maps  $\calL^n$ null sets to $\calH^n$ null sets.

\item $\Gamma$ is $n$-rectifiable.

\item 
For $\calH^n$ a.e. point $W(x)\in \Gamma$, with $x\in \Lambda_w$,
\beq
T_{W(x)}\Gamma = \mbox{Im}(DW(x))
\label{aptx}\eeq

\item $G_w$ is an i.m. rectifiable $n$-current represented by integration over 
$\Gamma$. Indeed, for every compactly supported $n$-form $\phi$ in $\Omega\times \R^\ell$,
\begin{equation}
G_w(\phi) = \int_\Gamma \langle \phi, \tau\rangle d\calH^n,\qquad\mbox{
where }\quad
\tau(x, \xi) = \frac { W_{x^1} (x)\wedge \ldots \wedge W_{x^n}(x)}{| W_{x^1}(x)\wedge \ldots \wedge W_{x^n}(x)|}.
\label{Gw.rep}\end{equation}

\item If $K$ is a compact subset of $\Omega$, then
$ \| G_w\| (K\times \R^\ell)  = \calH^n( \Gamma \cap (K\times \R^\ell) )<\infty$,
where $\|G_w \|$ denotes the total variation measure associated to $G_w$.

\end{enumerate}
\label{L.gms}
\end{lemma}

\begin{proof}
It follows from assumption \eqref{ww1kplus1} that 
$w$ is a.e. approximately differentiable,
and all minors of $Dw$ are locally integrable.
These are exactly the hypotheses of 
results in Giaquinta {\em et.}\,{\em al.} \cite{gms},
see  in particular sections 3.1.5 and 3.2.1
which  establish all the conclusions of the lemma.
\end{proof}

Under the conditions of Lemma \ref{L.gms},
the set $\Gamma$ which carries $G_w$ 
can differ from the actual graph
$\{( x, w(x)) : x\in \Omega\}$ by a set of positive $\calH^n$ measure;
see for example \cite{MalyMartio}.
As we show below, it is nonetheless true that the current
$G_w$ associated to $\Gamma$ has no boundary in $\Omega\times \R^\ell$.
For this we need the full strength of assumption \eqref{ww1kplus1}; for Lemma \ref{L.gms} above, it in fact suffices to assume that $w\in W^{1,k}_{loc}$.

\begin{lemma}
If $w$ satisfies \eqref{ww1kplus1} and  $G_w$ 
is the $n$-current defined in \eqref{Gw.def}, then
\beq
\partial G_w = 0\qquad\mbox{ in }\Omega\times \R^\ell .
\label{MA1}\eeq
\label{L.Gw}\end{lemma}

\begin{remark} The Lemma implies that if $u$ is a scalar function
and $w=Du$ satisfies \eqref{ww1kplus1}, then $u$
is a Monge-Amp\`ere function, see \cite{fu1, J2010}.
\end{remark}

\begin{proof}
We must check that 
\begin{equation}
0 = G_w( d(\phi\, dx^\alpha\wedge d\xi^\beta))
=
 G_w( \phi_{x^i} dx^i\wedge dx^\alpha\wedge d\xi^\beta) + 
 G_w( \phi_{\xi^j} d\xi^j\wedge dx^\alpha\wedge d\xi^\beta) 
\label{no.bdy}\end{equation}
for all $\phi\in C^\infty_c(\Omega\times \R^\ell)$ and $\alpha,\beta$ such that  $|\alpha|+|\beta|=n-1$.
The terms on the right-hand side have the form
\beq
\int_\Omega \phi_{x^i}(x, w) \cdot (\mbox{minor of order $|\beta|$})
+
\int_\Omega \phi_{\xi^i}(x, w) \cdot (\mbox{minor of order $|\beta|+1$}).
\label{nobdy2}\eeq
If $|\beta|\ge k+1$ then the assumption that rank$(Dw)\le k$ a.e.
implies that all such terms vanish, and hence that \eqref{no.bdy} holds. If $|\beta|\le k$, 
then let $w_q$ be a sequence of smooth functions converging to $w$
in $W^{1,k+1}_{loc}(\Omega, \R^\ell)$. 
For each $w_q$, \eqref{no.bdy} holds (with $w$ replaced by $w_q$).  Also, all minors of $Dw_q$ 
appearing in \eqref{nobdy2} have order at most
$k+1$, and hence
converge in $L^1_{loc}$ to the corresponding minors of $Dw$.
And we can arrange after passing to a subsequence that
\[
\left. \begin{array}{l}
\phi_{x^i}(x, w_q(x)) \to \phi_{x^i}(x, w(x))\\
\phi_{\xi^j}(x, w_q(x)) \to \phi_{\xi^j}(x, w(x))
\end{array}
\right\}\qquad
\mbox{$\calL^n$ a.e. $x$, \  as $q\to \infty$ }
\]
for all $i$ and $j$.
These terms are also pointwise bounded uniformly in $q$ (by $\|\nabla\phi\|_\infty$).
We can thus send $q\to \infty$ to conclude that \eqref{no.bdy} holds for $w$.
\end{proof}

Below, we write $J_k p_v$
for the $k$-dimensional Jacobian (in the sense of \cite{federer} 3.2.22)
of $p_v :\Gamma\to \R^\ell_v$, the point being that 
we implicitly restrict
the domain of $p_v$ to $\Gamma$.
Similarly, for $A\subset \R^\ell_v$, we understand
$p_v^{-1}(A)$ to mean $\{(x,\xi) \in \Gamma : \xi\in A\}$.

We can now prove Proposition \ref{P.weakfol}. In doing so, we establish a number of
additional facts that we record here:


\begin{lemma} Assume that $w$ satisfies \eqref{ww1kplus1} and let   $G_w$, $\Gamma_v$ and $\Gamma_h$ be defined, respectively,  
as in  \eqref{Gw.def}, \eqref{Gammav.def} and \eqref{Gammah.def}.
Then there exist measurable mappings $\tau_v : \Gamma_v\to \Lambda_k\R^\ell_v$ and
$\tau_h:p_v^{-1}(\Gamma_v)\to \Lambda_{n-k} (\R^n_h)$ such that
$\tau_v$ and $\tau_h$ are a.e.
unit simple multivectors orienting $T_\xi \Gamma_v$
and $T_{(x,\xi)}(\{\xi\} \times \Gamma_h (\xi))$, and
\begin{equation}
G_w( \chi \, d\xi^\beta\wedge  \psi) = \int_{\Gamma_v} 
H_\xi(\psi)
 \langle d\xi^\beta, \tau_v \rangle \ \chi  \   d\calH^k
\label{Gw.decomp}\end{equation}
for $\beta\in I(k,\ell) $,  $\psi\in \calD^{n-k}(\Omega\times \R^\ell_v)$ and $\chi\in C^\infty(\R^\ell)$, where
\begin{equation}
H_\xi(\psi) := \int_{\{\xi \}\times \Gamma_h(\xi)} \langle \psi,  \tau_h \rangle \, d\calH^{n-k}\qquad\mbox{ for  } 
\psi\in \calD^{n-k}(\Omega\times \R^\ell).
\label{Hxi.def}\end{equation}
\label{L.weakfol}
\end{lemma}

\begin{proof}[Proof of Proposition \ref{P.weakfol} and Lemma \ref{L.weakfol}]
{\bf 1}. Given that $\Gamma$ is rectifiable, see Lemma \ref{L.gms},
the measurability and rectifiability of $\Gamma_v$ are immediate consequences of \cite{federer} 3.2.31, and
then the {\em a.e.} measurability and rectifiability of $\Gamma_h(\xi)$ follow directly from \cite{federer} 3.2.22(2). 

Next, the coarea formula 
 \cite{federer} 3.2.22(3) states that for any $\calH^n\rest \Gamma$-integrable function $g$, 
\[
\int_\Gamma g \ J_k p_v \ d\calH^n \ = \ \int_{\Gamma_v}\left( \int_{ p_v^{-1}\{\xi\}} g \ d\calH^{n-k} \right) d\calH^k.
\]
It follows that
\begin{equation}
J_{k} p_v(x,\xi) >0 
\ \ \quad\mbox{ $\calH^{n-k}$ a.e. in $\Gamma_h(\xi)$, 
for $\calH^k$ a.e. $\xi\in \Gamma_v$.}
\label{Jp_ae1}\end{equation}
Moreover,
\begin{equation}
\quad T_\xi\Gamma_v = p_v(T_{(x,\xi)}\Gamma) = \mbox{Im}(Dw(x)),
\quad\mbox{ $\calH^{n-k}$ a.e. in $\Gamma_h(\xi)$,
for $\calH^k$ a.e. $\xi\in \Gamma_v$,}
\label{Jp_ae2}\end{equation}
using \cite{federer} 3.2.22(1) for the first equality,  and \eqref{aptx} for the second.

{\bf 2}. Let $\tau_v:\Gamma_v\to\Lambda_k\R^\ell_v$ be any fixed measurable
unit simple $k$-vectorfield that orients $T_\xi\Gamma_v$ a.e..
We will construct 
$\calH^n$-measurable $\tau_h:p_v^{-1}(\Gamma_v)\to \Lambda_{n-k} (\R^n_h)$ 
characterized (up to null sets) by the identity
\begin{equation}
\langle d\xi^\beta\wedge dx^\alpha, \tau(x,\xi)\rangle =
J_k p_v(x, \xi) \ \langle d\xi^\beta, \tau_v(\xi)\rangle \ 
\langle dx^\alpha, \tau_h(x,\xi)\rangle  
\label{tauh.char}\end{equation}
for all multiindices such that $|\beta| = n-|\alpha|=k$, where
$\tau$ was  defined in \eqref{Gw.rep}.
In fact, since $\tau_v$ and $\tau$ are measurable, this identity automaticaly the measurability of $\tau_h$.

To prove \eqref{tauh.char}, we fix some point $(x,\xi)\in p_v^{-1}\Gamma_v$
such that  rank$(Dw(x))=~k$ and \eqref{aptx} holds. These conditions hold
$\calH^n$ a.e. by \eqref{Jp_ae1} and Lemma \ref{L.gms}.
We will find $\tau_h$ by first selecting  a basis $\{ b_i\}_{i=1}^n$ for $\R^n_h$ with a number of good
properties, and then defining
\begin{equation}
\tau_i := DW(x) b_i, \ \ i=1,\ldots, n,\qquad\qquad \tau_h := \tau_{k+1}\wedge \ldots \wedge \tau_n.
\label{taui.def}\end{equation}
In view of \eqref{aptx},  any such $\{ \tau_i\}_{i=1}^n$ is a basis for
$T_{(x,\xi)}\Gamma$.
We choose $\{b_i\}$ to satisfy the following:
\begin{itemize}
\item  $\{ b_i\}_{i=k+1}^n$ are an orthonormal basis for $\ker(Dw(x))$.
\item $\{ b_i \}_{i=1}^k$ are orthogonal to $\ker(Dw(x))$, and are chosen so that $\{ \tau_i\}_{i=1}^k$
are orthonormal. 
\item $b_1,\ldots, b_k$ are ordered so that  $Dw(x) b_1 \wedge \ldots \wedge Dw(x) b_k$
is a positive multiple of $\tau_v(\xi)$.
\item  $\{ b_1,\ldots, b_n\}$ is
positively oriented with respect to the standard basis $\{ e_1,\ldots, e_n\}$.
\end{itemize}

The first two conditions can be satisfied since rank$(Dw(x))=k$.
The third condition can be achieved due to
 \eqref{Jpos_ae}, by changing the sign of $b_1$ if necessary.  
Having fixed $\{b_1,\ldots, b_k \}$, we can adjust the sign of $b_{k+1}$
to arrange the final condition.

We now verify \eqref{tauh.char}.
Note that  $\tau_i = DW(x) b_i = (b_i, Dw(x) b_i) \in \R^n_h\times \R^\ell_v$.
It follows that $\tau_i = (b_i,0)$ for $i > k$, and hence that $\{\tau_i\}_{i=1}^n$
are orthonormal. This and the ordering
of $\{b_1,\ldots, b_n\}$ imply that 
$\tau_1\wedge\ldots \wedge \tau_n =  \tau(x,\xi)$.

Also, it is a fact that $J_kp_v = |p_v\tau_1\wedge \ldots \wedge p_v \tau_k| $; this 
is a straightforward consequence of the defintion of the Jacobian.
Since $|\tau_v(\xi)|=1$ and $p_v\tau_i = Dw(x)b_i$, 
the ordering of $b_1,\ldots, b_k$ implies that
\[
\tau_v(\xi) = \frac{ p_v\tau_1\wedge \ldots\wedge p_v\tau_k }{ |p_v\tau_1\wedge \ldots\wedge p_v\tau_k|}
= \frac { p_v\tau_1\wedge \ldots\wedge p_v\tau_k} {J_k p_v(x,\xi)}.
\]
Since $p_v\tau_i = 0$ for $i>k$, it follows that
\begin{align*}
\tau(x,\xi) &=
\tau_1\wedge\ldots \wedge  \tau_n \\
&=
(p_h\tau_1+ p_v\tau_1)\wedge\ldots \wedge (p_h\tau_k+ p_v\tau_k)\wedge \tau_h
\\
&=
J_k p_v(x,\xi) \ \tau_v \wedge \tau_h
+ \mbox{(terms involving at most $k-1$ vertical vectors)}.
\end{align*}
Then the claim \eqref{tauh.char} follows by letting $d\xi^\beta\wedge dx^\alpha$ act by duality on both
sides of the above expression, since
\[
\langle d\xi^\beta \wedge dx^\alpha , \mbox{terms involving at most $k-1$ vertical vectors} \rangle = 0.
\]

{\bf 3}. We will now show that  if $|\beta| = n-|\alpha|\ge k$, then
\begin{equation}
\int_\Gamma \langle\, \phi \, d\xi^\beta\wedge dx^\alpha, \tau\rangle \  d\calH^n \ =
\int_{p_v^{-1}\Gamma_v} \langle\, \phi \, d\xi^\beta\wedge dx^\alpha, \tau\rangle  \ d\calH^n\ 
\qquad\mbox{ for  $\phi\in C^\infty_c(\Omega\times \R^n)$.}
\label{oldLem3.2}\end{equation}
This is clear  if $|\beta|= n-|\alpha| >k$, in which case both sides vanish.
For $|\beta|=k$, this follows from a classical argument, dating
back at least to  Fu \cite{fu1}, which we recall for the convenience of
the reader. First, we rewrite the left-hand side in terms of slices 
$\langle G_w, q_\beta,  \cdot \rangle$ of 
$G_w$ by level sets of $q_\beta$, where  $q_\beta(x,\xi) = (\xi^{\beta_1},\ldots, \xi^{\beta_k})\in \R^k$.
This leads to
\begin{equation}
\int_\Gamma \langle\, \phi \, d\xi^\beta\wedge dx^\alpha, \tau\rangle d\calH^n \ 
=  \ G_w( d\xi^\beta\wedge \phi\, dx^\alpha) \ 
 = \  \int_{\R^k} \langle G_w, q_\beta, y\rangle(\phi \, dx^\alpha) \ dy.
\label{ssl1}\end{equation}
Fix some   $i\in \{1,\ldots, \ell\}$. We will write 
 $q_i(x,\xi) = \xi^i$ and $q_{\beta, i}(x, \xi) = (q_\beta(\xi), \xi^i)\in \R^{k+1}$.
We claim that 
\begin{equation}
\Big\langle \langle G_w, q_\beta, y\rangle , q_i , s\Big \rangle = 0\qquad\quad\mbox{ for 
 {\em a.e. }}(y,s)\in \R^k\times \R.
\label{ssl0}\end{equation}
To see this, note that that for $\calL^{k+1}$ a.e. $(y,s)\in \R^{k}\times \R$,
\[
\Big\langle \langle G_w, q_\beta, y\rangle , q_i , s\Big \rangle
=
\langle G_w, q_{\beta,i}, (y,s)\rangle 
\]
(see \cite{federer} 4.3.5). Then basic properties of
slicing imply  that for any $\psi\in \calD^{n-k-1}(\Omega\times \R^\ell_v)$
and $\chi \in C^\infty_c(\R^k\times\R)$,
\[
\int_{\R^k\times \R} \langle G_w, q_{\beta,i}, (y,s)\rangle (\psi) \ \chi(y,s) \ dy\ ds
=
G_w( \chi\circ q_{\beta,i} \ d\xi^\beta \wedge d\xi^i \wedge \psi)  \overset{\eqref{nulll}} = 0.
\]
It follows that for every $\psi$ as above, 
\[
\Big\langle \langle G_w, q_\beta, y\rangle , q_i , s\Big \rangle(\psi)  = 0
\qquad\mbox{ for  {\em a.e. }}(y,s)\in \R^k\times \R.
\]
Then \eqref{ssl0} follows by considering a countable dense subset of $\calD^{n-k-1}(\Omega\times \R^\ell_v)$.

Now according to 
 Solomon's Separation Lemma  (Lemma 3.3 of \cite{solomon}),
 it is a consequence of \eqref{ssl0} that for $\calL^k$ a.e. $y$, every
indecomposable component of  $\langle G_w, q_\beta, y\rangle$  is carried by a level set of $q_i$.
Since this holds for all $i$, we infer that for a.e $y$, every indecomposable
component of $\langle G_w, q_\beta, y\rangle $ is carried by $p_v^{-1}\{\xi\}$
for some $\xi\in \R^\ell$. 
From general properties of slicing, each such indecomposable
component can be represented by integration with respect to $\calH^{n-k}$
over $p_v^{-1}\{\xi\}$. In particular, for each such indecomposable
component, $\calH^{n-k}(p_v^{-1}\{\xi\}) >0$, so
$\xi \in \Gamma_v$. Hence  $\langle G_w, q_\beta, y\rangle$ is carried by $p_v^{-1}\Gamma_v$.
We combine this fact with \eqref{ssl1} to deduce \eqref{oldLem3.2}.

 {\bf 4}. We now prove \eqref{Gw.decomp}.  
Thus, for  $\beta \in I(k,\ell),  \psi\in \calD^{n-k}(\Omega\times \R^\ell_v)$
and $\chi \in C^\infty(\R^\ell_v)$, we find from \eqref{Gw.rep}, \eqref{tauh.char}, \eqref{oldLem3.2}
and the coarea formula \cite{federer} 3.2.22 that
\begin{align*}
G_w( \chi \, d\xi^\beta\wedge  \psi ) 
&= \int_{p_v^{-1}\Gamma_v}  \langle d\xi^\beta\wedge  \psi, \tau \rangle \chi \, d\calH^n \\
&=
\int_{p_v^{-1}\Gamma_v} \langle\psi, \tau_h(x,\xi)   \rangle\langle d\xi^\beta, \tau_v(\xi)  \rangle  J_k p_v(x,\xi)
\chi(\xi) \, d\calH^n \\
&= \int_{\Gamma_v} 
 \left(\int_{ p_v^{-1}\{\xi\}} \langle \psi,  \tau_h \rangle d\calH^{n-k}\right)
 \langle d\xi^\beta, \tau_v \rangle \ \chi  \   d\calH^k
\end{align*}
This is \eqref{Gw.decomp}. 

{\bf 5}. Since $\partial G_w=0$ in $\Omega\times \R^\ell$, it follows from \eqref{Gw.decomp} that
\[
\int_{\Gamma_v} \partial H_\xi(\psi) \ \langle d\xi^\beta, \tau_v \rangle \chi(\xi) \, \calH^k = 0
\]
for all $\psi \in \calD^{n-k}(\Omega\times \R^\ell), \chi\in C^\infty(\R^\ell)$, and 
$\beta \in I(k,\ell)$. For every such $\psi$, ii follows that $\calH_\xi(\psi) = 0 $ for $\calH^k$
a.e. $\xi\in \Gamma_v$. By considering a countable dense subset of $\calD^{n-k}(\Omega\times \R^\ell)$, we conclude that 
\begin{equation}
\partial H_\xi = 0 \ \  \mbox{ in $\Omega\times \R^\ell$},  \qquad\mbox{ for } 
\calH^k \ \  a.e.  \ \ \xi\in \Gamma_v.
\label{slice.bdy}\end{equation}
This in turn implies that for $\calH^k$ a.e. $\xi\in \Gamma_v$,  $\tau_h(x,\xi)$ orients the approximate tangent space 
at $(x,\xi)$
to the rectifiable set $\{\xi\}\times \Gamma_h(\xi)$ for
$\calH^{n-k}$ a.e. $x\in \Gamma_h(\xi)$. 
Projecting this statement onto the horizontal component, and recalling 
and the choice of $\{\tau_i\}$ in Step 1 above, we deduce that
\[
T_x\Gamma_h(\xi) = \mbox{span} \{ p_h\tau_i\}_{i=k+1}^n = \mbox{span} \{ b_i\}_{i=k+1}^n  =  \ker(Dw(x)) .
\]
This completes the proof of  \eqref{Jpos_ae}, recalling that we have already verified
\eqref{Jp_ae2}.

{\bf 6}. Finally, comparing \eqref{Gw.def} and \eqref{Gw.decomp}, 
\[
\int_{\Lambda_w} \phi(x, w(x)) M^\beta_{\bar\alpha}(Dw) \ dx
= \pm\int_{\Gamma_v} 
 \left(\int_{ \{\xi\} \times \Gamma_h(\xi)}  \phi(x,\xi) \langle dx^\alpha,  \tau_h \rangle d\calH^{n-k}\right)
 \langle d\xi^\beta, \tau_v \rangle   \   d\calH^k
\]
if $|\beta|=n-|\alpha| = k$, for $\phi\in C^\infty_c(\Omega\times \R^\ell)$.
By an approximation argument, this also holds for
$\phi\in L^\infty(\Omega\times \R^\ell)$ with compact support.
Also, we may replace $\Gamma_v$ by $\Gamma_v^*$,  defined in \eqref{Gvstar.def}, since it follows from what we have
already proved that the latter has full $\calH^k$ measure in $\Gamma_v$.
We deduce that for any compact set $K\subset \Omega\times \R^\ell$, if we define
\[
\Omega^k_{\alpha,\beta, K} := 
\{ x\in \Lambda_w :  (x, w(x))\in K, \ M^\beta_{\bar \alpha}(Dw(x))  \ne 0 \}
\]
then
\[
\calL^n \left(\Omega^k_{\alpha,\beta, K} \setminus \cup_{\xi \in \Gamma_v^* }\Gamma_h(\xi)
\right) = 0 .
\]
Since 
\[
\Omega^k = \bigcup_{{|\beta|  \ = n-  |\alpha|=k} }  \ \  \bigcup_{ K\mbox{\scriptsize compact}}\Omega^k_{\alpha,\beta, K} \ , 
\]
and indeed this can be written as a countable union via a suitable sequence
of compact sets $\{K_j\}_{j=1}^\infty$, this implies \eqref{weakfoliate}.
\end{proof}

\section{Dense weak flat foliation} \label{fullrank}

The main result of this section is the following.

\begin{proposition}
Assume that $\Omega$ is a bounded, open subset of $\R^n$, and that
\beq
w\in W^{1,k+1}_{loc}(\Omega), \quad\quad
\mbox{rank}\,(D w) \le k \ \ \mbox{a.e.}
\label{w2kplus1}\eeq
for some $k\in \{ 1,\ldots, n-1\}$, 
and
\begin{equation}
w = (Du^1,\ldots, Du^q) \mbox{ for some }q\ge 1.
\label{gradient}\end{equation}
Then $w$ is densely weakly $(n-k)$-flatly foliated.
\label{P.dwff}\end{proposition}

This will be a straightforward consequence
of the following lemma, which gives a more detailed description
of $w$ in the set $\Omega^k$
in which $Dw$ has maximal rank $k$, see \eqref{Omegak.def}.

\begin{lemma}
Assume that $w$ satisfies the hypotheses of Proposition \ref{P.dwff}.

Then for $\calL^n$ a.e. $x\in \Omega^k$,
$w^{-1}\{ w(x) \}$ coincides, up to a $\calH^{n-k}$ null set, 
with a countable union of $(n-k)$-planes in $\Omega$, all of them parallel to
$\ker(Dw(x))$.

In particular, for $\calL^n$ a.e. $x\in \Omega^k$,
$w$ is $\calH^{n-k}$ {\em a.e.} constant on the
$n-k$-plane in $\Omega$
that passes through $x$ and whose tangent space
is $\ker(D w(x))$.
\label{P.weakdev}\end{lemma}
 
This is essentially proved in \cite{J2010} in the case $k=1, n=2$.


Note that for $w\in W^{1,k+1}_{loc}$, 
the set of points that fail to be Lebesgue points of $w$ has dimension
less than $n-k$, as discussed in Section \ref{background},
so the conclusions of the proposition make sense.

The proof of Lemma uses the geometric measure theory
results of the previous section to give a rigorous version
of the formal argument sketched in the introduction.
It is the only point in this paper at which we use the gradient
structure \eqref{gradient} of $w$.

In the proof we will identify $\R^n_h$ and $\R^\ell_v$ via the natural
isomorphism $e_i \leftrightarrow \ep_i$.

\begin{proof}[Proof of Lemma \ref{P.weakdev}]
{\bf 1}. We  fix  $\xi\in \Gamma_v^*$, defined in \eqref{Gvstar.def}, and we  first claim that
\beq
T_x\Gamma_h(\xi)  
\mbox{ is $\calH^{n-k}$ a.e. constant for $x\in \Gamma_h(\xi)$.}
\label{tsconstant}\eeq
Indeed, since $ D^2 u^i(x)$ is symmetric for every $i$,
 at $\calH^{n-k}$ a.e. $x\in \Gamma_h(\xi)$ we have
\begin{align*}
T_x\Gamma_h(\xi) 
\overset{\eqref{Jpos_ae}}  
=  \ker(D w(x))
\overset{\eqref{gradient} }= 
\cap_{i=1}^{q} \ker(D^2 u^i(x)) 
= 
\cap_{i=1}^{q}  [ \mbox{Im}(D^2 u^i(x))]^\perp.
\end{align*}
Moreover, if we write $P^i: (\R^n)^q\to \R^n$ to denote 
orthonormal projection of $\R^{n q} = (\R^n)^q$ onto 
the $i$th copy of $\R^n$, then $D^2 u^i(x) = P^i\circ Dw(x)$.
Thus
\[
\mbox{Im}(D^2 u^i(x))
= 
\mbox{Im} P^i\circ Dw(x) = P^i( Im(Dw_i)) \overset{\eqref{Jpos_ae}} = 
P^i( T_\xi\Gamma_v).
\]
The term on the right depends only on $\xi$, so \eqref{tsconstant} follows from
the
previous two identities.

{\bf 2}. 
For $\xi\in \Gamma_v^*$, we will write $T(\xi) := \cap_{i=1}^j [P^i(T_\xi\Gamma_v)]^\perp = T_x\Gamma_h(\xi)$ for a.e. $x \in \Gamma_h(\xi)$.
We next claim that 
\begin{equation}
\mbox{ if $\xi \in \Gamma_v^*$, then $\Gamma_h(\xi)$ is a union of $(n-k)$-planes in $\Omega$,
all  parallel to $T(\xi)$.} 
\label{planes}\end{equation}
Since the current $H_\xi$  from Proposition \ref{P.weakfol} is represented by integration over $\{\xi \}\times \Gamma_h(\xi)$, it suffices to show that
every indecomposable component of $H_\xi$ is supported on
exactly a set of the form
$\{\xi\}\times P$, where 
$P$ is an $(n-k)$-plane in $\Omega$ with tangent space $T(\xi)$.

This  follows from  \eqref{tsconstant} and the fact that $\partial H_\xi =0 $ in $\Omega\times\R^n$,
by classical arguments that we have already seen in 
the proof of Proposition \ref{L.weakfol}.
In detail,  by changing coordinates we may arrange that 
$T_{x}\Gamma_h(\xi) = \mbox{span}\{e_1,\ldots, e_{n-k}\}$ for a.e. 
$x\in \Gamma_h(\xi)$. 
Since $H_\xi$ is carried by $\{\xi\} \times \Gamma_h(\xi)$, it follows that for 
$H_\xi ( \phi \wedge df) = 0$ for every $n-k-1$-form $\phi$ with compact support
in $\Omega$, whenever $f$ has the form $f(x) = x^j$
for some $j\in \{n-k+1,\ldots, n\}$.
In this situation, Solomon's Separation Lemma  (Lemma 3.3 of \cite{solomon}) states that every
indecomposable component of $H_\xi$ is carried by a level set of $f$. It follows that
every indecomposable piece
of $H_\xi$ is contained in an   $n-k$ plane in which $x^j$ is constant
for all $j = n-k+1,\ldots, n$ (in the coordinates we have chosen, which depended on
$\xi$.)
described above. This completes the proof of \eqref{planes}.

{\bf 3}. Now the conclusions of the lemma follow directly from
\eqref{planes}, the definition \eqref{Gammah.def} of
$\Gamma_h(\xi)$, which implies in particular that $w$ is a.e. 
constant in each of these sets, and
\eqref{weakfoliate},  which asserts that $\cup_{\xi\in \Gamma_v^*} \Gamma_h(\xi)$
contains almost every point of $\Omega^k$.
\end{proof}

Having Lemma  \ref{P.weakdev}  at hand, the proof that $w$ is densely weakly flatly foliated
is straightforward. 

\begin{proof}[Proof of Proposition \ref{P.dwff} ]
{\bf 1}. We recall from Definition \ref{dWk-foliated} that the definition of 
densely weakly flatly foliated involves
a partition of $\Omega$ into sets $F_j$
such that  $\Omega_j := \cup_{m=0}^j F_m$ is open for every $j$, and satisfying
a property 
recalled in 
\eqref{dwff_j} below.
We define these sets  as follows. As before,
\[
\Omega^k := \{ x\in \Omega : \mbox{$x$ is a Lebesgue point of $w$ and $Dw$, and
$\mbox{rank}\,(D^2 u(x)) = k $} \}.
\]
We also let $\Omega_k = \Omega$, and for $j\in \{k-1, \ldots, 0\}$, we recursively
define (working downwards)
\begin{align}
\Omega_{j} &
= \Omega_{j+1} - \bar \Omega^{j+1}
\label{Omegaj1}\\
\Omega^{j} 
&=  \{ x\in  \Omega_{j} :
 \mbox{$x$ is a Lebesgue point of $Du$ and $D^2u$, and
$\mbox{rank}\,(D^2 u) = j $} \},
\label{Omegaj2}\end{align}
Finally, we set
\begin{equation}
F_j := \bar \Omega^j \cap \Omega_j \ = \ \Omega_j \setminus\Omega_{j-1} .
\label{Fj.def}\end{equation}
This indeed defines a partition of $\Omega$ such that every $\Omega_j$ is open, as required.

Note that by our convention $F_k= \bar{\Omega}^k$. 

We must show that  for every $j\in \{0,\ldots, k\}$,
\begin{multline} 
\mbox{ for every $x$ in a dense subset of $F_j$, there exists at least one $n-j$-plane $P$  in $\Omega_j$} \\
\mbox{such that $x\in P$ and 
$w$ is $\calH^{n-j}$ a.e. constant on $P$.}
\label{dwff_j}\end{multline}
Observe  for every $j \le k$, 
$\Omega_j$ is open, 
and $w\in W^{1, j+1}_{loc}(\Omega_j;\R^\ell)\subset W^{1,k+1}_{loc}(\Omega;\R^\ell)$, with rank$(Dw) \le j$
a.e. in $\Omega_j$. In other words, $w|_{\Omega_j}$ satisfies \eqref{w2kplus1} with
$k$ replaced by $j$, and hence Lemma \ref{P.weakdev} holds, with $k$ replaced by $j$
in $\Omega^j\subset\Omega_j$.
It follows that 
\begin{multline} 
\mbox{ for every $x$ in a full measure subset of $\Omega^j$, there exists at least one $n-j$-plane $P$  in $\Omega_j$} \\
\mbox{such that $x\in P$ and 
$w$ is $\calH^{n-j}$ a.e. constant on $P$.}
\label{fmwff_j}\end{multline}
Since $\Omega^j$ is manifestly dense in $F_j$, to deduce \eqref{dwff_j}
from \eqref{fmwff_j} it suffices to prove that
every full measure subset of $ \Omega^j$ is in fact dense in 
$\Omega^j$.

To see this, consider some $x_0\in \Omega^j$,
and 
fix $\delta>0$ such that 
rank$(A) \ge j$ for all matrices with $|A- Dw(x_0)|<\delta_0$.
Then for every $r>0$ such that $B_r(x_0)\subset \Omega_j$, since $x_0$
is a Lebesgue point of $w$ and $Dw$, the set
\[
\{x\in B_r(x_0) : x\mbox{ is a Lebesgue point of $w$ and $Dw$, and $|Dw(x)-Dw(x_0)|<\delta_0$} \}
\]
has positive measure. Since rank$(Dw)\le j$ a.e in $B_r(x_0)\subset \Omega_j$, 
the above set intersects $\Omega^j$ in a set of positive measure.
Since $x_0$ and $r$ were arbitrary, this completes the proof of \eqref{dwff_j}.
\end{proof}

\section{Pointwise weak developability } \label{S:3}

In this section we will prove the following statement, which is an important step in establishing Theorem 
\ref{T.degenerateMA}.

\begin{proposition}
Assume that
\beq
w\in W^{1,p}_{loc}(\Omega; \R^\ell), \quad\quad
\mbox{rank}\,(D w) \le k \ \ \mbox{a.e.}
\label{w2k+1b}\eeq
for some $k\in \{ 1,\ldots, n-1\}$ and some $p>k$. 
If  $w$ is densely weakly $(n-k)$-flatly foliated,
then $w$ is pointwise weakly $(n-k)$-flatly foliated.
\label{P.wd}\end{proposition}

\begin{remark} 
In view of Definition \ref{defi-develop}, we could say that Propositions \ref{isom.-is-degenerate},
\ref{P.dwff} and \ref{P.wd} 
together imply a pointwise weak developability result for $W^{2,k+1}(\Omega;\R^{n+k})$ isometric immersions, and
also for
such $u\in W^{2,k+1}$ such that
rank$(D^2 u) \le k$ a..e.
\end{remark}

The Proposition will follow from a couple of lemmas.

\begin{lemma}
Assume that $k,n$ are integers such that $1\le k < n$.
Let $U$ be an open subset of $\R^{n-k}$, and for $r>0$ let
$S := U\times B^{k}_r$
for some $r>0$.

Assume that $w\in W^{1,p}(S; \R^\ell)$ for some $p>k$,
and for $i=1,2$ let $\zeta_i:U\to B^k_s$ be continuous functions.
Then (writing points in $S $ in the form $x = (y,z)$ with $y\in U, z\in B^{k}_s$)
\[
\left(\int_{ U}  |w(y, \zeta_1(y))  -  w(y ,\zeta_2(y))|^p  dy\right)^{1/p}
\le C \| w \|_{W^{1,p}(S )} \| \zeta_1 - \zeta_2\|_{L^\infty(U)}^\alpha
\]
for $\alpha = 1 - \frac{k}p$, for a  constant $C$ depending only on $k$ and $p$.
\label{L.improve1}\end{lemma}

\begin{proof}
We compute 
\begin{align*}
\| w \|_{W^{1,p}(S )}^p
& \ge
\int_{U} \| w (y, \cdot)\|_{W^{1,p}(B^k_r)}^p \ dy\\
&\ge
C^{-1}\int_{U}\frac{| w(y, \zeta_1(y)) - w(y,\zeta_2(y))|^p }{|\zeta_1(y) - \zeta_2(y)|^{\alpha p}} \ dy
\end{align*}
by the ($k$-dimensional) Sobolev Embedding, from which we also know that the constant
$C$ depends only on $p$ and $k$ and in particular is independent of $r$.
\end{proof}

Our next lemma will be used again in Section \ref{S:4}.

\begin{lemma} \label{closure} 
Assume that $\Omega$ is a bounded, open subset of $\R^n$, and that
$w \in W^{1,p}(\Omega, \R^\ell)$  for some $ p>j\in \{ 1,\ldots, n-1\}$ and some $\ell$. 

Assume also that $x_0\in \Omega$, and that there exists a sequence of points
$(x_m)\subset \Omega$ and values $(\xi_m)\in \R^\ell$
such that $x_m\to x_0$ as $m\to \infty$, and $w = \xi_m$ at $\calH^{n-j}$ a.e.
point on an $(n-j)$-plane $P_m$ in $\Omega$ that contains $x_m$.

Then exists, and $w=\lim_{m\to\infty}\xi_m$ at $\calH^{n-j}$ a.e. point on some $n-j$ plane $P$ in $\Omega$ that contains $x_0$. (In particular, $\lim_{m\to\infty}\xi_m$ exists.)
\end{lemma}


\begin{proof}
Let $\xi_m\in \R^\ell$ denote the value of $w$ on $\calH^{n-j}$ a.e. point of
$P_m$, and let $\PP_m$  denote the $(n-j)$-plane such that $P_m$ is a 
connected component of $\PP_m\cap \Omega$.

Since the Grassmannian of unoriented $(n-j)$-dimensional subspaces in $\R^n$
is compact,  we may  assume, after passing to subsequences (still
labelled $ (P_m), (\xi_m)$) that there is a $(n-j)$-plane $\PP$
passing through $x_0$ such that $\PP_m\to \PP$ in the Hausdorff distance
on $B_R(0)\subset\R^n$ as $m\to \infty$, for every $R>0$. Now let $P$ be the
$(n-j)$-plane in $\Omega$ consisting of the connected component of 
$\PP\cap \Omega$ that contains $x_0$.

We may arrange, after a translation and a rotation,
that $x_0= 0$
and $\PP = \R^{n-j}\times \{0\}$,
and we write $\R^n = \R^{n-j}_y\times \R^j_z$ as in Lemma \ref{L.improve1}. 
Fix a connected, relatively open set $U\subset P$, containing $x_0$ and having compact closure in
$\Omega$. Then there exists an open ball $B_r^j$ such that 
$S := U\times B_r^j \Subset \Omega$.  The convergence $\PP_m\to \PP$ implies that
for every sufficiently large $m$, there is an affine function
$\zeta_m:U \to B^j_r$
such that $\PP_m\cap S = \{ (y, \zeta_m(y)) : y\in U \}$,
and moreover that $\| \zeta_m\|_{L^\infty(U)} \to 0$
as $m\to \infty$.

Also, for $m$ large enough that $x_m\in S$, 
we have that $P_m\cap S$ is nonempty, and hence (since $S\subset \Omega$ is convex
and $P_m$ is a connected component of $\PP_m\cap \Omega$) that 
$\PP_m\cap S = P_m\cap S\subset P_m$. So $w=\xi_m$ $\calH^{n-j}$ a.e. in $\PP_m\cap S$, 
and
by applying  Lemma \ref{L.improve1} to $\zeta= 0$ and $\zeta_m$, we find that
\begin{align*}
\int_{U} |w(y,0) - \xi_m|^{p} dy\ &= \ \int_{U} |w(y,0) - w(y, \zeta_m(y))|^{p} dy\\
&\le C \| w\|^{p}_{W^{1,p}(S)} \| \zeta_m\|^{\alpha p}_{L^\infty(U)} \to 0\quad\mbox{as $m\to\infty$}, 
\end{align*} where $\alpha= 1 -\frac jp$. It follows that there exists some $\xi\in \R^\ell$ such that $\xi_m\to \xi$, and moreover
that $w(\cdot,0)= \xi$ a.e. on $U$. Since $U$ was arbitrary, 
it follows that $w= \xi$ at $\calH^{n-j}$ a.e. point of $P$. 

\end{proof}

Now we complete the

\begin{proof}[Proof of Proposition \ref{P.wd}]
By assumption, $\Omega$ is partitioned into sets $F_j$, $j=0,\ldots, n-k$
such that $\Omega_j := \cup_{m=0}^j F_m$ is open for every $j$,
and in addition, there is a dense subset of $F_j$ in which every point
is contained in a $n-j$-plane in $\Omega_j$ on which $w$ is $\calH^{n-j}$ a.e.
constant.

To prove the Proposition (with the same partition $(F_j)$ of $\Omega$),
it suffices to show that {\em every} point in $F_j$ 
is contained in a $n-j$-plane in $\Omega_j$ on which $w$ is $\calH^{n-j}$ a.e.
constant.
This follows directly from Lemma \ref{closure}, since every point in $F_j$
satisfies the hypotheses of the lemma, with $\Omega$ replaced by $\Omega_j$.
 \end{proof}

\begin{remark}
We note in passing that a slightly more careful version of
the above argument would prove the following statement: For {\em every} $x\in  \Omega^j$ as defined in \eqref{Omegaj2}, $w$ is $\calH^{n-j}$ {\em a.e.} constant on the
$n-j$-plane in $\Omega_j$ that passes through $x$ and whose tangent space
is $\ker(D w(x))$, and the constant value is equal to $w(x)$. 
 \end{remark} 



 

\section{Strong developability} \label{S:4}

In this section we prove the following

\begin{proposition}
Assume that $\Omega$ is an open subset of $\R^n$, and that
that $w\in W^{1,p}_{loc}(\Omega;\R^\ell)$ for some $p\ge \min\{2k, n\}$. If $w$ 
is pointwise weakly $(n-k)$-flatly foliated, then $w$ is continuous.
As a result, if $P$ is any $n-j$-plane in $\Omega_j$  (as in Definition \ref{Wk-foliated})  
on which $w$ is $\calH^{n-j}$ a.e. constant, then in fact $w$ is
constant on $P$. In particular, $w$ is $(n-k)$-flatly foliated. 
\label{P.sd}\end{proposition}

For the convenience of the reader, the proof will be split in a series of Lemmas which will follow and will be completed in Lemma \ref{L.w2}. 
This will complete the proof of Theorems \ref{T.isom_immersion} and \ref{T.degenerateMA},  which follow immediately from combining  Propositions \ref{P.dwff},  \ref{P.wd} and
\ref{P.sd} and, for Theorem \ref{T.isom_immersion} only, Proposition  \ref{isom.-is-degenerate} 
as well.

The following examples shows that the condition $p\ge \min\{2k, n\}$ cannot be weakened,
at least for certain values of $n$ and $k$.

\begin{example}
Consider the map $w:\R^4\to S^2\subset \R^3$ defined by
\[
w(x) = H( \frac x{|x|}) \quad\mbox{ if }x\ne 0,\quad
w(0)= 0,
\]
where $H:S^3\to S^2$ is the Hopf fibration. Recall that every level set of $H$ has the form
$\{ (z, \zeta)\in \C^2\cong \R^4 : |z|^2 + |\zeta|^2 = 1, \  \alpha z  = \beta \zeta \}$
for some fixed $\alpha,\beta \in \C$ (one of which can always be taken to equal $1$). 
From this one easily checks that 
$w$ is a $2$-plane passing through the origin, and that the intersection of
any two level sets is $\{0\}$. Thus, $w$ is
pointwise weakly $(n-k)$-flatly foliated (see Definition \ref{Wk-foliated})
with $n=4, k=2$ and $F_2 = \R^4$, $F_0 = F_1  = \emptyset$, and $w\in W^{1,p}$ for
all $p< 4 = \min \{2k, n\}$. But clearly $w$ is not continuous.

This example shows the hypothesis $p\ge\min\{ 2k, n\}$ of Proposition \ref{P.sd}
cannot be weakened when $n = 2k = 4$.
\label{ex:Hopf1}\end{example}

\begin{example}
Next, for $n\ge 5$ define $w_1:\R^n\to  \R^3$ by
$w_1(x^1,\ldots, x^n) = w(x^1,\ldots, x^4)$
where $w$ is the function from the above example. 
Then  $w_1$ is pointwise weakly $(n-k)$-flatly foliated
with $k=2$ and $F_2 = \R^n$, $F_0 = F_1  = \emptyset$. Also,  $w_1\in W^{1,p}_{loc}$ for
all $p< 4 = \min \{2k, n\}$.  But again $w_1$ is not continuous.

So the condition $p\ge \min\{ 2k,n\}$ cannot be weakened whenever $k=2$ and $n >4$.
\label{ex:Hopf2}\end{example}

\begin{example}
One can construct  a function  similar to that of Example \ref{ex:Hopf1} when $n = 2k = 8$ or $16$ by using  
Hopf fibrations $S^7\to S^4$ and $S^{15}\to S^8$, 
and similarly a function similar to the one in Example  \ref{ex:Hopf2} when $n > 2k = 8$ or $16$ 
It follows that the condition $p\ge \min\{2k,n\}$ cannot be weakened whenever
$k = 4$ or $8$ and $n\ge 2k$.
\label{ex:Hopf3}\end{example}

\begin{remark} One can check that the $w:\R^n\to \R^\ell$ constructed in the above examples
are not gradients of scalar functions. In fact we conjecture
that if we add to Proposition \ref{P.sd} the assumption that $w = Du$ for some scalar function
$u$, then the conclusions of the proposition should still be true if we merely assume $p\ge k+1$.
\end{remark}

The next lemma, whose proof is very similar to that of Lemma \ref{L.improve1}, still only needs the minimal 
regularity assumptions $p>k$.

\begin{lemma}
Assume that $k,n$ are integers such that $1 \le k<n$. Let $\Omega$
be an open subset of $\R^n$, and assume that $w\in W^{1,p}_{loc}(\Omega; \R^\ell)$ for some $p>k$.
Finally, assume that $P$ is an $n-k$-plane in $\Omega$ such that $w= \xi$ a.e. on $P$ for some $\xi\in \R^\ell$.

If $x\in P$ is a Lebesgue point of $|Dw|^p$, then $x$ is a Lebesgue point of $w$, and
$w(x) = \xi$.
\label{L.lebpt}\end{lemma} 

\begin{proof}
We may assume after a translation and a rotation that $P$ is a connected component of $\Omega \cap (\R^{n-k}\times \{0\})$, and that $x = 0$.
Fix $R>0$ such that $B^{n-k}_{R}\times B^k_R \subset \Omega$,
and let $\alpha = 1-\frac kp$.
Then for any positive $r<R$, writing $[f]_\alpha$ to denote the $\alpha$-H\"older seminorm,
\begin{align*}
\barint_{B^{n-k}_r\times B^k_r} |w - \xi|^p \, dy\,dz\ 
&=
\barint_{B^{n-k}_r} 
\left( \barint_{B^{k}_r} | w(y,z) - w(y,0)|^p dz
\right) dy\\
&
\le
\barint_{B^{n-k}_r} 
\left(
 \barint_{B^{k}_r} |z|^{p\alpha} [w(y,\cdot)]_\alpha^p dz
\right) dy.
\end{align*}
Also, by the $k$-dimensional Sobolev embedding,
\begin{align*}
\barint_{B^{k}_r} |z|^{p\alpha} [w(y,\cdot)]_\alpha^p dz
\le
C r^{\alpha p - k} \int_{B^k_r} |Dw(y,z)|^p \ dz = Cr^{\alpha p} \barint _{B^k_r} |Dw(y,z)|^p \ dz
\end{align*}
with a constant $C$ independent of $r$.
Thus
\[
\barint_{B^{n-k}_r\times B^k_r} |w - \xi|^p  \, dy\,dz\
\le r^{\alpha p} 
\barint_{B^{n-k}_r\times B^k_r} |Dw|^p \, dy\,dz\ .
\]
Since $x$ is a Lebesgue point of $|Dw|^p$, the right-hand side is bounded by
$Cr^{p\alpha}$ for all small $r$, proving the lemma.

\end{proof}

The restriction $p\ge \min \{2k, n\}$ in Proposition \ref{P.sd} arises from
the following lemma.

\begin{lemma}
Assume that $\Omega$ is an open subset of $\R^n$
and that  $w\in W^{1,p}_{loc}(\Omega, \R^\ell)$ for some $\ell$ and some $p\ge 1$. Suppose that for $i=1,2$, there exist
values $\xi^i \in \R^n$,  planes $P_i$ in $\Omega$
of dimension $n-k$ such that 
\[
P_1\cap  P_2\ne \emptyset,
\quad\quad\quad\quad\mbox{ and } \ \ w = \xi^i, \ \ \ \mbox{ $\calH^{n-k}$ a.e. in }P_i
\]
for $i=1,2$. If $p\ge \min\{ n, 2k\}$  then $\xi_1 = \xi_2$.
\label{L.no_sing_k}\end{lemma}

\begin{proof}
{\bf 1}. We first consider the case $2k< n$.

Let $x_0\in \Omega \cap P_1\cap P_2$.
Any two planes of dimension $n-k$ that intersect at a point must intersect
along a plane of dimension $n-2k$.
We may assume after a translation that $x_0$ is the origin, and after a 
rotation that $P_1\cap P_2 = \R^{n-2k}\times \{0\}$.
We write $y$ and $z$ respectively to denote points in $\R^{n-2k}$ and in $\R^{2k}$, 
and we fix $r$ and $s$ such that $B_r^{n-2k}\times B_s^{2k} \subset \Omega$.
Then for $\calH^{n-2k+1}$ a.e. $(y,\sigma) \in B_r^{n-2k}\times (0,s)$,
\[
\mbox{ess\,osc}_{\{y\}\times  \partial B^{2k}_\sigma}|w| \ge |\xi_1-\xi_2|,
\]
so that by the Sobolev embedding theorem,
\[
|\xi_1-\xi_2|^{2k} \le C  \sigma  \int_{\{y\}\times  \partial B^{2k}_\sigma} |Dw|^{2k} \ d\calH^{2k-1}.
\]
Thus
\begin{align*}
\int_{B^{n-2k}_r\times B^{2k}_s}
|Dw|^{2k}
&= 
\int_{B^{n-2k}_r} \int_0^s \int_{\{y\}\times  \partial B^{2k}_\sigma} |Dw|^{2k} \ d\calH^{2k-1} \ d\sigma\ dy \\
&\ge
c |\xi_1-\xi_2|^{2k}  \int_{B^{n-2k}_r} \int_0^s \frac 1\sigma d\sigma \ dy .
\end{align*}
The left-hand side is finite, so 
it follows that $|\xi_1-\xi_2| = 0$.

{\bf 2}. The case $2k \ge n$ is similar but easier.
Here, all we can say about any two $n-k$-planes with
nonempty intersection is that their intersection must contain
a point $x_0$. Hence, the essential oscillation of $w$ on a.e. small sphere centered at
$x_0$ is bounded below by $|\xi_1-\xi_2|$, and as a result
\beq
\int_{B_s^n(x_0)}|Dw|^n = \int_0^s \int_{\partial B^n_\sigma(x_0)} |Dw|^n \ge c |\xi_1-\xi_2|^n \int_0^s \frac 1 \sigma\ d\sigma.
\label{case2k.gtr.n}\eeq
We conclude as before that $|\xi_1-\xi_2|=0$.
\end{proof}

\begin{remark}If $2k \ge n$, then a small modification of the above proof shows that the conclusion remains true
if we assume $w = \xi_1$ a.e. in $P_1$ and that  
$w = \xi_2$ at $\calH^1$ a.e. point of a connected, relatively open subset
$U\subset P_2$, with $P_1\cap \bar U\ne \emptyset$. Indeed,  these hypotheses imply
the existence of an open line segment containing $x_0$ on which $w= \xi_1$ a.e., and a second open line
segment with an endpoint at $x_0$ on which $w=\xi_2$ a.e., and these conditions imply that
the essential oscillation of $w$ on a.e. small sphere centered at
$x_0$ is bounded below by $|\xi_1-\xi_2|$, allowing us to conclude as in \eqref{case2k.gtr.n}.

\label{no.sing.bis}\end{remark}
 
Our next result follows rather easily from the above two lemmas.

\begin{lemma}
Assume that $w\in W^{1,p}_{loc}(\Omega;\R^\ell)$ for some $p\ge \min\{2k, n\}$. If $w$ 
is pointwise weakly $(n-k)$-flatly foliated,
then there exists a function $\overline{w}:\Omega\to \R^n$ such that
\begin{equation}
\overline{w}|_{F_j} \mbox{ is continuous  for every }j\in \{0,\ldots, k\}
\label{w1.1}
\end{equation} and 
\begin{equation}
\overline{w} = w \mbox{ {\em a.e.} in $\Omega$}.
\label{w2.1}
\end{equation}
In particular, 
for every $x\in F_j$, there is an $n-j$ plane in $\Omega_j$ containing $x$
on which $\overline{w}= \overline{w}(x)$ everywhere, where $F_j$ and $\Omega_j$ are given as in Definition \ref{Wk-foliated}.
\label{L.w1}\end{lemma}

\begin{proof}
{\bf 1}. We define $\overline{w}$ by requiring that
\[
\overline{w}(x) = \xi\mbox{ if $x\in F_j$ and $w = \xi$ a.e. on some $n-j$-plane $P$ in $\Omega_j$ passing through $x$}.
\]
We claim that that $\overline{w}$ is well-defined. Towards this end, 
note that every $x$ belongs to a unique $F_j$ by \eqref{Fj1.2}
and hence by \eqref{Fj3.2} belongs to at least one $n-j$-plane in $F_j$ on which $w$ is a.e.
constant. Then by Lemma \ref{L.no_sing_k}, the values of $w$ on any two such planes must agree
a.e., so the claim follows.

{\bf 2}. It follows from the definition of $\overline{w}$ and Lemma \ref{L.lebpt} that $w=\overline{w}$ at every
Lebesgue point of $|Dw|^p$, which implies \eqref{w2.1}.

{\bf 3}.
To  verify that \eqref{w1.1} holds, assume assume toward a contradiction that $\overline{w}|_{F_j}$
is not continuous at some point $x_0\in F_j$. 
Then there exists a sequence $(x_m)$ in $F_j$ such that
\[
|x_m - x_0| <  \frac 1m, \quad\quad
|\overline{w}(x_m) - \overline{w}(x_0)| \ge c_0
\]
for some $c_0>0$. Let $\xi_m := \overline{w}(x_m)$, and let $P_m$ be a $n-j$ plane in $\Omega_j$
such that $\overline{w} = \xi_m$ on $P_m$.
Then
\beq
P_m\cap B_{1/m}(x_0)\ne \emptyset\quad\quad\quad
w = \xi_m \mbox{ a.e.  on }P_m.
\label{c1.1}\eeq
Then Lemma \ref{closure} implies that
there exists some exactly $(n-j)$-plane $P'$ in $\Omega_j$ and some $\xi'\in \R^n$
such that
\[
x_0\in P', \quad\quad
\xi_m \to \xi', \quad \quad
\mbox{ and }w = \xi' \ \ \calH^{n-j} \ \mbox{ a.e. on }P'.
\]
The definition of $\overline{w}$ implies that $\overline{w}(x_0) = \xi'$. This however is impossible, since $\xi_m\to \xi'$
and $|\xi_m - \overline{w}(x_0)|\ge c_0$ for  all $m$.
This contradiction shows that $\overline{w}|_{F_j}$ is continuous on $F_j$.

\end{proof}

Our next goal is to show that the function $\overline{w}$ found above is continuous in all of $\Omega$. This
will directly imply the continuity of $w$,  and hence will conclude the proof of our main results.

\begin{lemma}
Assume that $w\in W^{1,p}_{loc}(\Omega;\R^\ell)$ for some $p\ge \min\{2k, n\}$ and that $w$ 
is pointwise weakly $(n-k)$-flatly foliated.
Let $\overline{w}$ be the function found in Lemma \ref{L.w1}. Then $\overline{w}$ is continuous in $\Omega$,
and as a result, $w$ is continuous in $\Omega$.
\label{L.w2}\end{lemma}

Before giving the proof, we recall that every $f\in W^{1,p}(\Omega,\R^\ell)$
is {\em p-quasicontinuous}, which means that 
for every $\e>0$, there exists an open set $O\subset \Omega$ such that Cap$_p(O)< \e$ and
$f|_{\Omega\setminus O}$ is continuous.
For the definition and the few properties of capacity that are needed  for our argument (e.g. the above statement)  refer to \cite {eg}, 
unless another reference is provided.

The idea of the proof below is to show that, given what we already know about $w$, if
it is discontinuous anywhere, then it must fail to be $p$-quasicontinuous, for $p=\min\{2k,n\}$,
which is impossible. That is, we will  argue  (in the more difficult case $2k<n$)
that, in view of \eqref{w1.1},
any discontinuity of $\overline{w}$ would involve the intersection of (the closure of)
portions of planes on which $\overline{w}$ is constant, one having dimension at least 
$n-k$ and the other dimension at least $n-k+1$. This would lead
to a discontinuity set for $w$ of dimension at least $n-2k+1$, along which the
discontinuity cannot be eliminated  by cutting out an open set of small enough
$p$-capacity, the point being that a set of $p$-capacity zero
has dimension strictly less than $n-2k+1$.

\begin{proof}[Proof of Lemma \ref{L.w2}]
First, since $\overline{w}=w$ a.e.,  if $\overline{w}$ is continuous, then every $x\in \Omega$ is a Lebesgue point 
of $w$, and the Lebesgue value at $x$ equals $\overline{w}(x)$. So $w=\overline{w}$ pointwise in $\Omega$,
and the continuity of $w$ follows. Thus we only need to show that $\overline{w}$ is continuous.

It is convenient to write $F_{\ge j} := \bigcup_{m\ge j} F_m$, and similarly $F_{>j} := \bigcup_{\ell>j} F_m = F_{\ge j+1}$.
With this notation, we will prove that  by  (downward) induction on $j$ that 
\beq
\mbox{ $\overline{w}|_{F_{\ge j}}$ is continuous for every 
$j\in \{k, \ldots, 0\}$}
\label{wcont}\eeq
which in particular will imply that $\overline{w}$ is continuous on $F_{\ge 0} = \Omega$.

From Lemma \ref{L.w1} we already know that \eqref{wcont} holds for $j=k$. 
Now  we assume by induction
that $\overline{w}|_{F_{>j}}$ is continuous for some nonnegative $j<k$, and
we prove that $\overline{w}|_{F{\ge j}}$ is continuous.

{\bf Step 1}. We first show that
\beq
\mbox{if  $P$ is an $n-j$-plane in $\Omega_j$ for  which $\overline{w} = \xi$  on $P$, \ then }
\mbox{ 
$\overline{w} = \xi $ on $\bar P \cap F_{>j} $.}
\label{wmainpoint}\eeq
This is a key point of the proof. In the case $2k\ge n$, this follows in a straightforward 
way from Remark \ref{no.sing.bis}, so we focus on the case $2k<n$.

{\bf Step 1a}.
Assume toward a contradiction that \eqref{wmainpoint} fails,
so that for some $n-j$-plane $P$ in $\Omega_j$
and  $x_0\in \bar P\cap F_{>j}$ such that
\beq
\mbox{ $\overline{w} = \xi$ on $P$, \ and \  
 $ \overline{w}(x_0) = \xi_0$, \qquad for some $\xi \ne \xi_0\in \R^\ell$. }
\label{w.contra}\eeq
Then $x_0\in F_i$ for some 
$i> j$, so there exists an $n-i$-plane $P_0$ in $\Omega_i$
such that $x_0\in P_0$ and $\overline{w} = \xi_0$ in $P_0$.

We may assume that 
\beq
P\cap P_0 = \emptyset
\label{disjointPs}\eeq
because if there exists some $y_0\in P \cap P_0$,
then since both $P$ and $P_0$ are relatively open, we could apply
Lemma \ref{L.no_sing_k} on a small ball containing $y_0$ to conclude that 
$\xi = \xi_0$.

We may also assume (after a translation) that $x_0=0$.
We write  $\PP$ and $\PP_0$ to denote the planes (of dimension $n-j $ and $n-i$ respectively) that contain $P$ and $P_0$, 
and we let $d$ denote the dimension of $\PP\cap \PP_0$, 
so that $d\ge n-i-j \ge n-2k+1$, recalling that $j<i \le k$.
Also, $d< n - i = \dim(\PP_0) < n-j$.

We can arrange by a suitable rotation that
\[
\PP 
= 
\R^{n-j}\times \{0\}\subset \R^n,\quad\quad\quad
\PP\cap \PP_0 
= \R^{d}\times \{0\} \subset \R^n.
\]
We will write points in $\R^n$ in the form
$x = (y, z)$ with $y\in \R^{d}$, $z\in \R^{n-d}$.

By the induction hypothesis, we may fix $r>0$ so small that $B^d_r\times B^{n-d}_r\subset \Omega_i$
and
\beq
|\overline{w}(x) - \xi| > \delta := \frac 12|\xi_0-\xi| \quad\quad \mbox{ for all  }  \ \ x\in (B^d_r\times B^{n-d}_r)\cap F_{>j}.
\label{minjump}\eeq
Let $B$ be a relatively  open ball in $P\cap (B^d_r\times B^{n-d}_r) $, and 
let $B_0$ denote the orthogonal projection of $B$ onto $\R^{d}\times \{0\}$, so that
$B_0$ is a relatively open subset of $B^d_r\times \{0\}$.

\medskip

{\bf Step 1b}.
We claim that for every $y\in B_0$, the restriction of $w$ to $\{y\}\times B^{n-d}_r$
is discontinuous.

This is a consequence of the following two facts, which we will prove below.
First,
\beq
\mbox{ $\forall y\in B_0$, \ \ $(\{ y\} \times B^{n-d}_r)\cap  \partial_\PP P$  is nonempty,}
\label{bdnonempty}\eeq
where $\partial_\PP P$ denotes the  boundary of $P$ in $\PP$.
Second, 
\beq
\mbox{  $w$  is discontinuous at every point of
$\partial_{\PP} P \cap (B^d_r\times B^{n-d}_r)$}.
\label{ess_discont}\eeq
(Recall that $w$ is identified with its precise representative, and that the complement of
the set of Lebesgue points has dimension less than $n-p-\ep$ for every $\ep>0$,
and in particular is a $\calH^{n-p+1}$ null set.)

To prove \eqref{bdnonempty}, we first note that
the definition of  $B_0$ implies directly that 
\beq
\mbox{$(\{ y\} \times B^{n-d}_r)\cap  P$  is nonempty\quad for $y\in B_0$.}
\label{dPP1}\eeq
Also, the definitions imply that 
\beq
\mbox{$B^d_r\times \{0\} \subset P_0$.}
\label{eek}\eeq
This is verified by noting that 
$P_0 \cap (B^d_r\times B_r^{n-d}) $ is nonempty, since $x_0 = (0, 0)\in
P_0$, and that in addition $P_0$ is a connected,
relatively open subset of $\PP_0\cap \Omega_i$.
Since $(B^d_r\times \R^{n-d}_r)\subset \Omega_i$, it follows that
$P_0$ contains $\PP_0\cap (B^d_r\times B^{n-d}_r)$, which implies
\eqref{eek}. 

From \eqref{eek} and \eqref{disjointPs} we see that $(y, 0)\not\in P$,
and hence  that
\beq
\mbox{$(\{ y\} \times B^{n-d}_r)\cap  (\PP\setminus P)$  is nonempty.}
\label{dPP2}\eeq
Since $P$ is a connected, relatively open subset of $\PP$,
the claim \eqref{bdnonempty} follows from \eqref{dPP1} and \eqref{dPP2}.

To prove \eqref{ess_discont}, fix $z\in \partial_\PP P\cap (B^d_r\times B^{n-d}_r)$,
and note that $z\in\Omega_i \setminus \Omega_j$, since $P$ is by definition
a connected component of $\PP\cap \Omega_j$, 
and $\Omega_j$ is open. Thus $z\in F_m$ for some
$j<m \le i$, and so there exists an $n-m$ plane $P_1$ in $\Omega_m$
containing $z$, and 
on which $w = \overline{w}(z)$   $\calH^{n-m}$ a.e..
So every ball around $z$ contains points at which $w = \overline{w}(z)$.
Similarly, \eqref{minjump} implies that 
every ball around $z$ contains points at which $w = \xi \ne \overline w(z)$.
Therefore \eqref{ess_discont} follows, completing Step 1b.

\medskip

We now establish \eqref{wmainpoint}. 
Since $w$ is $p$-quasicontinuous, for any $\ep>0$,
there exists a set $S$ such that 
the restriction of $w$ to $\Omega \setminus S$ is continuous,
and $\mbox{Cap}_p(S) < \ep$. By Step 1b, the orthogonal projection of $S$
onto $\R^d\times \{0\}$ must contain the open ball $B_0$.
Note that  $p$-capacity is not increased by orthogonal projection, e.g. by \cite[Theorem 3]{meyers} 
(See also \cite[Chapter 5]{ah} for further discussion of this type of results). Therefore it follows that 
$\mbox{Cap}_p(B_0) < \ep$ for every $\ep>0$, and hence that
$\mbox{Cap}_p(B_0) = 0$. This however is false, as a set with zero
$p$-capacity has $H^s$ measure $0$ for every $s>n -p$, and
the dimension $d$ of $B_0$ satisfies
$d \ge n -2k +1 > n-p$.
So we have proved  \eqref{wmainpoint}.
%
\medskip

{\bf Step 2}. We now use \eqref{wmainpoint} to prove the continuity of $\overline{w}$ on $F_{\ge j}$.

Clearly $F_{\ge j}$ is partitioned as  $F_{>j} \cup F_j$. Since $\Omega_j$ is open
and $F_j = \Omega_j \cap F_{\ge j}$, we see that $F_j$ is relatively open
and  $F_{> j}$ relatively closed  in $F_{\ge j}$. 
Thus, in view of the induction hypothesis and Lemma \ref{L.w1}, 
it suffices to check that if $x_0\in F_{>j}$
and $(x_m)$ is a sequence in $F_j$ converging to $x_0$, then 
$\overline{w}(x_m) \to \overline{w}(x_0)$.

Thus we fix some $x_0\in F_i$ for some $i>j$, 
and we assume toward a  contradiction, that there is a sequence $(x_m)$ in $F_j$ such that
\[
x_m\to x_0,\quad\quad |\overline{w}(x_m) - \overline{w}(x_0)| \ge c_0 >0\quad\mbox{ for all }m.
\]
The definition of $\overline{w}$ implies that
there exists an $n-i$-plane $P$ in $\Omega_i$ such that $x_0\in P\subset \Omega_i$, $\overline{w} = \overline{w}(x_0)$
everywhere on $P$, and $w = \overline{w}(x_0)$ almost everywhere on $P$.
It further implies that for each $x_m$, there exists a $n-j$-plane $P_m$ in $\Omega_j$ such that $x_m\in P_m$, and  on which
$\overline{w} = \xi_m := \overline{w}(x_m)$ everywhere, and $w= \xi_m$ almost everywhere.

For each $m$ we write $\PP_m$ to denote the $n-j$-plane such 
that $P_m$ is a connected component of $\Omega_j\cap \PP_m$.
We now consider two cases. 

{\bf Case 1}. There exists some $\delta>0$ and a subsequence $(m_q)$
such that 
$\PP_{m_q}\cap B_\delta (x_0) \subset \Omega_j $ for every $q$.

If this holds, then it follows from Lemma \ref{closure}, with
$\Omega$ replaced by $B_\delta(x_0)$, that 
there exists some $n-j$-plane in $B_\delta(x_0)$
that contains $x_0$, and on which
$w= \lim \xi_{m_q}$ a.e..  This however would imply that $\overline{w}(x_0) = \lim \xi_{m_q}$, which is impossible.

{\bf Case 2}. Next we suppose that Case 1 does not hold.

Then for every $q$ there is some
$m_q$ such that
\[
\PP_{m_q}\cap B_{1/q}(x_0) \not \subset \Omega_j.
\]
For $q$ large enough that $B_{1/q}(x_0)\subset \Omega  = \Omega_j\cup F_{>j}$,
it must then be the case that
$\bar P_{m_q} \cap F_{>j} \cap B_{1/q}(x_0)\ne \emptyset$. 
Let $y_{m_q} \in \bar P_{m_q} \cap F_{>j} \cap B_{1/q}(x_0)$.

By Step 1,  we know that $\overline{w}(y_{m_q}) = \overline{w}(x_{m_q})$.

Also, by construction, $y_{m_q}\to x_0$ as $q\to \infty$.
Then, since $y_{m_q}\in F_{>k}$ for every $q$, it follows from 
the induction hypothesis that 
$\overline{w}(x_0) = \lim_{q\to \infty} \overline{w}(y_{m_q}) = \lim_{q\to \infty} \overline{w}(x_{m_q})$, which
is impossible in view of the choice of the sequence $(x_m)$. Hence $\overline{w}$ is continuous as claimed.
\end{proof}


\begin{thebibliography}{10}


 


\bibitem{ah} D.R. Adams and L.I. Hedberg, \textit{Function Spaces and Potential Theory}, 
Grundlehren der mathematischen Wissenschaften (Book 314), Springer, 1999.

\bibitem{bori1}
Yu. F. Borisov,
\textit{The parallel translation on a smooth surface. III.},
Vestnik Leningrad. Univ. {\bf 14} (1959) no. 1, 34--50.


\bibitem{bori2}
Yu. F. Borisov,
\textit{Irregular surfaces of the class $C^{1,\beta}$ with an analytic metric},
 (Russian) Sibirsk. Mat. Zh. {\bf 45} (2004), no. 1, 25--61; 
translation in Siberian Math. J. {\bf 45} (2004), no. 1, 19--52.
 
  
  \bibitem{Cart} 
  \newblock E. Cartan, 
  \newblock {\em  Bull. Soc. Math. France.}  {\bf 46}, 125, (1919); {\bf 48}, 132, (1920).
  
 
 
\bibitem{ChL}  S.S. Chern and R.K. Lashof, \textit{On the
Total Curvature of Immersed Manifolds}, 
American Journal of Mathematics,  {\bf 79}, (1957) No. 2., 306--318. 
\bibitem{CDS} 
S. Conti, C. De Lellis and L. Sz\'ekelyhidi Jr., \textit{$h$-principle and rigidity for $C^{1,\alpha}$ isometric embeddings},
To appear in the Proceedings of the Abel Symposium 2010.



\bibitem{eg}
L.C. Evans and R.F. Gariepy, \textit{Measure theory and fine properties of functions},
Studies in Advanced Mathematics, CRC Press, 1992.

 
\bibitem{federer}
\newblock H.~Federer,
\textit{Geometric Measure Theory},
\newblock Springer-Verlag, Berlin, 1969.




\bibitem{FM} I. Fonseca and J. Mal\'y, \textit{From Jacobian to Hessian: distributional form and relaxation},  
Riv. Mat. Univ. Parma, {\bf 7},  (2005), 4*, 45--74.

\bibitem{FJMgeo} G. Friesecke, R. James and S. M\"uller,
\textit{A theorem on geometric rigidity and the derivation of nonlinear
plate theory from three dimensional elasticity}, Comm. Pure. Appl. Math.,
{\bf 55} (2002), 1461--1506.

\bibitem{fu1}
\newblock J.H.G.~Fu,
\textit{Monge-Amp\`ere functions I}, 
\newblock {\em Indiana Univ. Math. Jour.}, {\bf 38}, no. 3, (1989), 745--771.


\bibitem{gms} 
\newblock M.~Giaquinta, G.~Modica and J.~Soucek, 
\textit{Cartesian Currents in the Calculus of 
Variations I}, 
\newblock Springer-Verlag, New York, 1998. 



\bibitem{HartmanNirenberg}
P. Hartman and L. Nirenberg,
\textit{On spherical image maps whose Jacobians do not change sign}, {Amer. J. Math.,}
{\bf 81} (1959), 901--920.

\bibitem{Ho0} P. Hornung, \textit{Approximating $W^{2,2}$ isometric immersions},
C. R. Math. Acad. Sci. Paris, {\bf 346} (2008), no. 3-4, 189--192.

\bibitem{Ho1} P. Hornung, \textit{Fine level set structure of flat isometric immersions},
Arch. Rational Mech. Anal., {\bf 199} (2011), 943--1014.

\bibitem{Ho2} P. Hornung, \textit{Approximation of flat $W^{2,2}$ isometric immersions
by smooth ones},
Arch. Rational Mech. Anal., {\bf 199} (2011), 1015--1067.

\bibitem{Je1} R.L. Jerrard, \textit{Some remarks on Monge-Amp\`ere functions}, Singularities in PDE and the calculus of variations, 
CRM Proc. Lecture Notes, {\bf 44}, (2008), 89--112. 

\bibitem{J2010} R.L. Jerrard, \textit{Some rigidity results related to Monge-Amp\`ere functions}, Canad. J. Math. {\bf 62}, no. 2, (2010), 320--354


\bibitem{kirchheimthesis}
B.~Kirchheim,
\textit{Geometry and {R}igidity of {M}icrostructures.}
Habilitation Thesis, Leip\-zig, (2001), Zbl pre01794210.

\bibitem{kirchhoff}
G.~Kirchhoff. \textit{{\"U}ber das gleichgewicht und die bewegung einer elastischen
  scheibe},
{\em {J}. {R}eine {A}ngew. {M}ath.},  (1850), 51--88.

\bibitem{kuiper} 
N. H. Kuiper, \textit{On $C^1$-isometric imbeddings. {I}, {II}.}, Nederl. Akad. Wetensch. Proc. Ser. A. {\bf 58},  
(1955) 545--556, 683--689. 


\bibitem{lepa_noneuclid} M. Lewicka and M. Pakzad,
    \textit{Scaling laws for non-Euclidean plates and the $W^{2,2}$
      isometric  immersions of Riemannian metrics}, ESAIM:
    Control, Optimisation and Calculus of Variations (2010).


\bibitem{LM-private} Z. Liu and J. Mal\'y, private communication. 

\bibitem{LP} Z. Liu and  M. R. Pakzad, 
\textit{Rigidity and regularity of co-dimension one Sobolev isometric immersions}, 
To appear in Ann. Scuola Norm. Sup. Pisa Cl. Sci. (5),   
DOI Number: 10.2422/2036-2145.201302\textunderscore 001, \url{http://arxiv.org/pdf/1302.0075v2.pdf}


\bibitem{MalyMartio}
{J. Mal\'y and O. Martio}, 
\textit{Lusin's condition (N) annd mappings of the class $W^{1,n}$,}
\newblock {\em J. Reine Angew. Math.}, {\bf 458}, (1995), 19--36.

\bibitem{meyers} N.G. Meyers, \textit{Continuity of Bessel potentials}, 
Israel Journal of Mathematics, {\bf 11}, (1972), Issue 3, 271--283.

\bibitem{MuPa2} S. M\"uller and M.R. Pakzad, \textit
{Regularity properties of isometric immersions},  Math. Z., {\bf 251} (2005), no. 2, 313--331  .



\bibitem{Nash2} J. Nash, \textit{$C^1$ isometric imbeddings}, Ann. Math., {\bf 60}, (1954), 383--396. 


\bibitem{Pak} M.R. Pakzad, \textit{On the Sobolev space of isometric immersions},
J. Differential Geom., {\bf 66},  (2004) no. 1, 47--69.


\bibitem{Po 56} A.V. Pogorelov, 
\textit{Surfaces with bounded extrinsic curvature} (Russian), Kharhov, 1956.

\bibitem{Po 73} A.V. Pogorelov, 
\textit{Extrinsic geometry of convex surfaces}, Translation of mathematical
monographs vol. 35, American Math. Soc., 1973.

\bibitem{solomon}
B.~Solomon,
\textit{new proof of the Closure Theorem for integral currents}, 
\newblock {\em Indiana Univ. Math. Jour.}, {\bf 33}, no. 3, (1984), 393--418.


\bibitem{VWKG} S.C. Venkataramani, T.A. Witten, E.M. Kramer and R.P. Geroch, 
\textit{Limitations on the smooth confinement of an unstretchable manifold},  J. Math. Phys. 
{\bf 41}, no. 7, (2000), 5107--5128. 

 
\bibitem{Ziemer} 
\newblock W.~P.~Ziemer
\textit{Weakly differentiable functions. Sobolev spaces and functions of bounded variation,} 
\newblock Graduate Texts in Mathematics, 120. Springer-Verlag, New York, 1989. 



\end{thebibliography}
\end{document}